\documentclass[10 pt, reqno]{amsart}
 
  \usepackage{amsmath}
    \usepackage{paralist}
  \usepackage{amsmath}
  \usepackage[colorlinks=true]{hyperref}
  \hypersetup{urlcolor=blue, citecolor=red}
\newtheorem{thm}{Theorem}[section]
\theoremstyle{plain}

\newtheorem{lem}[thm]{Lemma}
\newtheorem{prop}[thm]{Proposition}

\theoremstyle{definition}
\newtheorem{assum}[thm]{Assumption}
\newtheorem{remark}{Remark}[section]

\numberwithin{equation}{section}

\renewcommand{\d}{\/\mathrm{d}\/}

\def\A{\mathcal{A}}

\def\C{\mathrm{C}}

\def\J{\mathcal{J}}
\def\B{B}
\def\R{\mathcal{R}}

\def\N{\mathbb{N}}

\def\no{\nonumber}

\def\U{\mathrm{U}}

\renewcommand{\d}{\/\mathrm{d}\/}

\makeatletter
\newcommand{\leqnomode}{\tagsleft@true}
\newcommand{\reqnomode}{\tagsleft@false}
\makeatother

\title[ H$^\infty$ stabilization for sabra Shell Model] 
      {Interior and H$^\infty$ feedback stabilization for sabra Shell model of turbulence}

\author[Tania Biswas and Sheetal Dharmatti]
{Tania Biswas and Sheetal Dharmatti}


 \email{sheetal@iisertvm.ac.in}
 \email{tanibsw91@gmail.com}

\thanks{$^*$ Corresponding author}
\begin{document}
\maketitle

\centerline{\scshape Tania Biswas}
\medskip
{\footnotesize
 \centerline{Department of Mathematics, IIT Delhi}
   \centerline{IIT Campus, Hauz Khas,}
   \centerline{ New Delhi, Delhi 110016}
} 

\medskip

\centerline{\scshape Sheetal Dharmatti$^*$}
\medskip
{\footnotesize
 \centerline{School of Mathematics,  IISER Thiruvananthapuram}
   \centerline{Maruthamala P O Vithura,}
   \centerline{ Trivandrum, Kerala, India 695551}
}

\bigskip

 \centerline{(Communicated by the associate editor name)}

\begin{abstract}
Shell models of turbulence are  representation of turbulence equations in Fourier domain. Various shell models are studied for their mathematical relevance and  the  numerical simulations which exhibit at most resemblance with turbulent flows. One of the mathematically well studied  shell model  of turbulence is  called sabra shell model.  This work concerns with two important issues related to shell model namely feedback stabilization and robust stabilization. We  first address stabilization problem related to sabra shell model of turbulence and prove that the system can be  stabilized via finite dimensional controller. Thus only finitely many modes of the shell model would suffice to stabilize the system.  Later we study  robust stabilization in the presence of the unknown disturbance and corresponding control problem by solving an infinite time horizon max-min control problem. We first prove the $H^ \infty$ stabilization of the associated linearized system and characterize the optimal control in terms of a feedback operator by solving an algebraic riccati equation. Using the same riccati operator we establish asymptotic  stability of the nonlinear system.
\end{abstract} 

\section{Introduction}
\setcounter{equation}{0}
Mathematical modelling of turbulence is  complicated, yet various theories and models are proposed in the literature for example  \cite{bohr}, \cite{donough}, \cite{frisch}, \cite{kandanoff}. Turbulent flows show large interactions at local levels/nodes. Hence it is suitable to model them in frequency domain or what is commonly known as Fourier domain. Shell models of turbulence are simplified caricatures of equations of fluid mechanics in wave-vector representation. They exhibit anomalous scaling and local non-linear interactions in wave number space. 
Shell models are well known as they retain certain features of Navier Stokes Equations. The spectral form of Navier Stokes Equations is the motivation  to study shell models. But, unlike spectral model of Navier Stokes Equations, shell models contain local interaction between the modes, that is interaction in the short range which is important in modeling turbulent phenomena. However the governing equations are derived by the necessity that the helicity and energy are conserved as in the case of Navier Stokes Equations.  Several shell models have been proposed in literature; refer to \cite{Peter} for more details.  The most popular and well studied shell model was proposed by Gledzer and was investigated numerically by Yamada and Okhitani, which is referred as the Gledzer- Okhitani -Yamada or GOY model in short \cite{gledzer}, \cite{yamada}. The numerical experiments performed by them showed that the model exhibits an enstrophy cascade and chaotic dynamics. In this work we consider a model known as sabra shell model, introduced in \cite{L'vov et al}.  The Sabra shell model of turbulence shows similar energy cascade as in the GOY model and is also suitable for therotical study of the model.
In \cite{PBT}, Constantin, Levant and Titi have studied this model analytically and have obtained existence  and uniqueness of the strong and weak solutions for the equations in appropriate spaces. In  \cite{PBT 1}, the same authors have  studied the global existence of weak solutions of the inviscid sabra shell model and have shown that these solutions are unique for some short interval of time. Moreover, they give a Beal-Kato-Majda type criterion for the blowup of solutions of the inviscid sabra shell model and show the global regularity of the solutions in the two-dimensional  parameters regime. 
We now brief  the model used in this paper and describe the problem of interest.
    
\subsection{Spectral form of NSE and shell model}
The spectral form of Navier Stokes Equations is a starting point of shell models. Consider the incompressible Navier Stokes Equation which is given by 
\begin{equation*}
\frac{du}{dt} +  (u\cdot \nabla)u = -\nabla p + \nu \Delta u + f
\end{equation*} 
with the continuity equation
\begin{equation*}
div \; u =0.
\end{equation*}
in the domain $\Omega \subset \mathbb{R}^d$ where $d=2$ or $3$. Here, $u$ denotes the velocity of the fluid, $p$ is the pressure and $f$ is the forcing term. To rewrite Navier Stokes Equations in spectral form we take Fourier transform  of  equation to get,
\begin{equation}\label{spec nvs}
 \frac{du_j(n)}{dt} = -i\left( \frac{2\pi}{L} \right) n_j \sum_{n'} \left( \delta_{il} - \frac{n_in_l'}{n^2} \right) u_i(n')u_l(n-n') - \nu k_n^2u_j(n) + f_j(n)
 \end{equation}
where $n$ and $n'$  are vectors in $\mathbb{R}^d$,
\begin{align*}
u_j(k) = \frac{1}{(2\pi)^3} \int \exp^{-ikx}u_j(x) dx 
\end{align*}
and the wave vectors $k(n) $ are given by $ k(n)= \frac{2\pi n}{L}$ [see \cite{Peter}]. \\
To describe the shell model, the spectral spaces are divided into concentric spheres of exponentially growing radius,
\begin{equation}\label{e2}
k_n=k_0\lambda ^n
\end{equation}
with fixed $\lambda>1$ and $k_0>0$. The one dimensional wave numbers are  denoted by $k_n$'s such that $ k_{n-1} < |k| < k_n$. The set of wave numbers contained in the $n^{th}$ sphere is called $n^{th}$ shell and $\lambda$ is called shell spacing parameter. The spectral velocity  $u_n$ is a kind of mean velocity, of the complex Fourier coefficients of the velocity in the $n^{th}$ shell. 

 
The equations of motion for the sabra shell model are given by
\begin{equation}\label{e1}
\frac{du_{n}}{dt}=i(ak_{n+1}u_{n+2}
  u_{n+1}^*+bk_nu_{n+1}u_{n-1}^*-ck_{n-1}u_{n-1}u_{n-2})-\nu k_n^2u_n+f_n
\end{equation}
for $n=1,2,3, \cdots$, with the convention that $u_{-1} = u_0 = 0$. The kinematic viscosity is represented by $\nu>0$ and $f_n$'s are the Fourier components of the forcing term. The nonlinear term defines the nonlinear interaction between the nearest nodes. The constants $a,b,c$ are chosen such that $a+b+c=0$. For more details one can refer to Chapter 3 of \cite{Peter}. 



Control problems associated with turbulence equations in general and shell models in particular, have not been studied widely. To our knowledge, there were no known results for the control problems associated with shell models of turbulence until in \cite{BD}, we have studied optimal control problems and invariant subspaces for the sabra shell model. However, stabilization results for shell models are completely open. In the current work our aim is to design a finite dimensional controller in the feedback form which will exponentially stabilize sabra shell model of turbulence.  Further we will study robust stabilization which will stabilize sabra shell model in the presence of unknown disturbance.

Stabilization results for the nonlinear parabolic partial differential equations have been actively studied for the past two decades.  The stabilization problem for Navier-Stokes equations have been studied for the case of control acting as a distributed parameter using infinite dimensional controller in \cite{barbu 1}, \cite{T} and using finite dimensional controllers in \cite{BT}, \cite{BT 1}, \cite{barbu ns}. Moreover, boundary stabilization of fluid flow problems have been extensively studied in \cite{fursikov1}, \cite{T1}, \cite{T}, \cite{raymond1} using infinite dimensional feedback controller.  The robust stabilization  for finite-dimensional control system are detailed in \cite{doyle} and for infinite dimensional system is developed in \cite{tadmor}, \cite{barbu}, \cite{van1}. Robust stabilization using frequency domain approach has been introduced in \cite{zames} for a finite dimensional system and also discussed in \cite{foais}, \cite{curtain}.  The $H^\infty$ stabilization for abstract parabolic systems with internal control is studied in \cite{benpaper}, \cite{basar}, \cite{bewley}, \cite{mejdo}  and  for Navier Stokes' systems in \cite{robust1}. In \cite{boudstab} authors study  the $ H^\infty$ boundary stabilization for Navier-Stokes equations.  

 In the current work we extend results of  \cite{barbu}, \cite{BT} for stabilization and \cite{robust1} for $H^\infty$  stabilization of sabra shell model of turbulence. The novelty of the work lies in finding finite dimensional controller which in the particular case of shell model translates to  proving that only finitely many modes will suffice to stabilize the system. Moreover,  stabilization problem and $H^\infty$ stabilization  for turbulence models and/or shell models are treated in this paper for the first time as per our knowledge.

The paper is organized as follows: We discuss the functional setting of the problem, important properties of the operators involved and the existence result from \cite{PBT} in the next section. We are reiterating few of the properties and important theorems from \cite{PBT} and \cite{PBT 1}. In section 3, we first prove existence of solution for steady state equation corresponding to sabra shell model and later part is devoted to prove internal stabilization of sabra shell model via finite dimensional feedback controller. In section 4 we study the robust stabilization of the model. For, we first consider linearized system and prove robust stabilization of it. Further we prove for nonlinear system, results hold provided initial data and disturbance are small enough.  

\section{Functional setting}
In this section we consider the functional framework considered in \cite{PBT} so that equation \eqref{e1} can be written in operator form in infinite dimensional Hilbert space. We look at $\{u_n\}$ as an element of $H=l^2(\mathbb{C} )$ and rewrite the equation \eqref{e1} in the following functional form by appropriately defining operators $A$ and $B$, 
\begin{equation} \label{e3}
\frac{du}{dt}+\nu Au+B(u,u)=f \ u(0)=u^0.
\end{equation}
For defining operators $A$ and $B$ we introduce certain functional spaces below.
For every $ u, v \; \in {H}$ the scalar product $(\cdot,\cdot)$ and the corresponding norm $|\cdot|$ are defined as,
\begin{equation*}
(u,v)=\sum\limits_{n=1}^{\infty} u_nv_n^*,\ |u|=(\sum\limits_{n=1}^{\infty} |u_n|^2)^{\frac{1}{2}}.
\end{equation*}
Let $(\phi_j)_{j=1}^{\infty}$ be the standard canonical orthonormal basis of $H$. The linear operator $A:D(A)\rightarrow H$ is defined through its action on the elements of the canonical basis of $H$ as
\begin{equation*}
A\phi_j=k_j^2\phi_j,
\end{equation*}
where the eigenvalues $k_j^2$ satisfy relation \eqref{e2}. The domain of $A$ contains all those elements of $H$ for which  $| Au| $ is finite. It is denoted by  $ D(A) $ and is a dense subset of $H$. Moreover, it is a Hilbert space when equipped with graph norm 
\begin{equation}
\|u\|_{D(A)}= |Au| \ \ \forall u\in D(A). \nonumber
\end{equation}

The bilinear operator $B(u,v)$ will be defined in the following way. Let $u,v\in H$ be of the form  $u=\sum_{n=1}^{\infty}u_n\phi_n$ and  $v=\sum_{n=1}^{\infty}v_n\phi_n$.
 Then,
\begin{equation*}
B(u,v)=-i\sum_{n=1}^{\infty}(ak_{n+1}v_{n+2}
 u_{n+1}^*+bk_nv_{n+1}u_{n-1}^*+ak_{n-1}u_{n-1}v_{n-2}+bk_{n-1}v_{n-1}u_{n-2})\phi_n.
\end{equation*}
With the assumption $u^0=u_{-1}=v_0=v_{-1}=0$ and together with the energy conservation condition $a+b+c=0$, we can simplify and rewrite $B(u,v)$ as
$$B(u,u)  = -i\sum_{n=1}^{\infty}(ak_{n+1}u_{n+2}
 u_{n+1}^*+bk_nu_{n+1}u_{n-1}^*-ck_{n-1}u_{n-1}u_{n-2})\phi_n$$

With above definitions of $A$ and $B$, \eqref{e1} can be written in the form $$\frac{du}{dt} +\nu  A u + B(u,u) = f , \ u(0) = u^0.$$  We now give some properties of $A$ and $B$.

Clearly, $A$ is positive definite, diagonal operator. Since $A$ is a positive definite operator, the powers of $A$ can be defined for every $s\in \mathbb{R}.$ 
For $ u=(u_1,u_2,...) \in H, $  define $ A^su=(k_1^{2s}u_1,k_2^{2s}u_2,...)$.\\
Furthermore we define the spaces  
$$V_s:=D(A^\frac{s}{2})={\{u=(u_1,u_2,...) \ \vert \ \sum\limits_{j=1}^{\infty} k_j^{2s}|u_j|^2<{\infty}\}}$$\\
which are Hilbert spaces equipped with the following scalar product and norm,\\
$$(u,v)_s=(A^{s/2}u,A^{s/2}v) \ \forall u,v \in V_s, \; \|u\|=(u,u)_s \ \ \forall u \in V_s.$$ 
Using above definition of the norm we can show that  $V_s \subset V_0=H \subset V_{-s}  \ \ \forall s>0$. Moreover, it can be shown that the dual space of $ V_s $ is given by $ V_{-s} $.
Domain of $A^{1/2} $ is denoted by  $V$ and is  equipped with scalar product $((u,v))=  (A^{1/2}u,A^{1/2}v) \ \ \forall u,v \in D(A^{1/2})$. Thus we get the  inclusion 
 $$V \subset H=H' \subset V,' $$
where $ V^\prime$, the dual space of $V$ which is  identified with  $D(A^{-1/2})$. The norm in $V$ is denoted by $\| \cdot \|$.  We denote by $ \langle\cdot,\cdot\rangle$ the action of the functionals from $ V' $ on the elements of $ V $. 
Hence for every  $ u\in V $, the $H$ scalar product of $f \in H$ and $ u\in V $ is  same as the action of $f$ on $u$ as a functional in $V'$.
$$_{V'}\langle f,u\rangle_V = (f,u)_H \ \ \forall f \in H, \ \forall u \in V.$$ 
So for every $u\in D(A)$ and for every $ v \in V$, we have $((u,v))=(Au,v)= \langle Au,v \rangle$ .\\
Since $D(A)$ is dense in $ V$ we can extend the definition of the operator $A:V \longrightarrow V' $ in such a way that $\langle Au,v \rangle = ((u,v)) \ \ \forall u,v \in V.$\\
In particular it follows that 
$$\|Au\|_{V'} = \|u\|_V \ \  \forall u \in V.$$

%

\begin{thm} \textbf{(Properties of bilinear operator $B$ )} \label{prop b}
\begin{enumerate}
\item $B:H\times V\longrightarrow H$ and $B:V\times H\longrightarrow H$ are bounded, bilinear operators. Specifically 
\begin{enumerate}
\item $|B(u,v)|\leq C_1|u|\|v\| \ \ \forall u\in H, v\in V$
\item  $ |B(u,v)|\leq C_2|v|\|u\| \ \ \forall u\in V, v\in H$
\end{enumerate} 
where \\
 $C_1=(|a|(\lambda^{-1}+\lambda)+|b|(\lambda^{-1}+1))$ \\ 
 $C_2=(2|a|+2\lambda|b|)$.  \\
 \item $B:H\times H\longrightarrow V'$ is a bounded bilinear operator and \\
 $\|B(u,v)\|_{V'}\leq C_1|u||v| \ \ \forall u,v\in H.$ \\
 \item $B:H\times D(A)\longrightarrow V$ is a bounded bilinear operator and for every $u\in H$ and $v\in D(A)$ \\
  $\|B(u,v)\|\leq C_3|u||Av|  \\
  C_3=(|a|(\lambda^{3}+\lambda^{-3})+|b|(\lambda+\lambda^{-2})).$  \\
 \item For every $u\in H$ and $v\in V$, $Re(B(u,v),v)=0.$ \\
 \item Let $u,v,w \in V.$ Denote $b(u,v,w)= \langle B(u,v),w \rangle$. Then 
 \begin{enumerate} 
 \item $b(u,v,w) = - b(v,u,w) $
 \item $b(v,u,w) = - b(v,w,u) . $
 \item $b(u,v,v)=0 .$ \\
\end{enumerate}
\item Let us denote $B(u)=B(u,u)$. Then the map $B: V \rightarrow V' $ which takes $u \longrightarrow B(u)$ is Gateaux differentiable. Moreover, for each $u  \in V$ the  Gateaux derivative of $B$ in the direction of $v \in V $  is denoted by $ B'(u) v : V \rightarrow V'$ and is given by
\begin{align}\label{bprime}
 B'(u)v &= B(u,v) + B(v,u) , \ \ \ \forall v\in V.\\
\text {and} \; \; \langle B'(u)v,w\rangle_{(V',V)} &= b(u,v,w) + b(v,u,w) \ \ \ \forall u, v, w \in V.  \nonumber
\end{align}
\item Let $B'(u)^* $ denote the adjoint of $B'(u)  $. Therefore for each $ v \in V $,   we have $\langle B'(u)v,w\rangle_{(V',V)}= \langle v,B'(u)^*w\rangle_{(V, V' )} \ \   \forall w\in V$. \\
Hence,  $ B'(u)^*w : V \rightarrow V'$ is  given by
$$ B'(u)^*w = -B(u,w) - B(w,u) \ \ \ \forall w\in V.$$
\end{enumerate}
\end{thm}

\begin{proof}
The proofs 1.-4. are given in \cite{PBT}. Refer Theorem 2.1.1 of \cite{BD} for the proof of 5.,6.,7.
\end{proof}


The existence and uniqueness for shell model of turbulence \eqref{e3} are studied in \cite{PBT}  mainly using Galerkin approximation and Aubin's Compactness lemma. We  state these theorems below. For proofs see \cite{PBT} Theorem 2 and Theorem 4.
\begin{thm}\label{weak}
Let $f \in L^2(0,T;V')$ and $u^ 0\in H$. Then there exists a unique weak solution $u \in L^{\infty}([0,T],H) \cap L^2([0,T],V)$ to \eqref{e3}. Moreover the weak solution $u \in C([0,T],H)$.  
\end{thm}
\begin{thm}\label{strong}
Let $f \in L^{\infty}([0,T],H)$ and $u^0 \in V$. Then there exists a unique strong solution $u \in C([0,T],V) \cap L^2([0,T],D(A))$ to \eqref{e3}. 
\end{thm}

\section{Internal Stabilization}
Consider the controlled sabra shell model of turbulence 
\begin{equation} \label{i1}
\frac{du}{dt}+\nu Au+B(u) = f+\B_1 U , \ u(0)=u^0,
\end{equation}
where $A$, $B$ are as defined in the previous section and $B_1 \in \mathcal{L}(H,H)$ and $U$ is the control variable. Further we have $A$ and $\B$ satisfy the following properties,
\begin{assum}
\begin{enumerate}
\item $-A$ generates a $C_0$-semigroup on $H$. \\
\item $B$ is G\'ateaux differentiable on $D(A)$, i.e., 
$$B'(u)z = \lim_{\lambda \rightarrow 0} \frac{B(u+\lambda z)- B(u)}{\lambda}$$
exist in $H$ for all $u, z \in D(A).$ [see 6, 7 of Theorem \eqref{prop b}]
\end{enumerate}
\end{assum}

Now we will study the internal stabilization of \eqref{i1} via finite dimensional feedback controller. For that let us consider the steady state equation given by 
 \begin{equation} \label{steady}
\nu Au_e+B(u_e,u_e)=f.
\end{equation}		
We prove the following existence theorem of the steady state system \eqref{steady}.
\begin{thm}
Let $f\in V'$. Then there exists a weak solution $u_e \in V$ for the steady state system \eqref{steady} with the weak formulation 
\begin{equation}\label{i2}
\nu \langle Au_e,v \rangle +\langle B(u_e,u_e),v\rangle =\langle f,v \rangle .
\end{equation} 
for all $v\in V$.

Moreover if we take $f \in H$, then $u_e \in D(A).$
\end{thm}
\begin{proof}
The proof follows as in the case of steady state Navier-Stokes equations by using Galerkin approximation technique, see \cite{temam}. We construct an approximate solution of \eqref{steady} and then pass to the limit.
 
Let $\{e_j\}_{j=1}^\infty$ be the eigen functions of the operator $A$, as a Galerkin basis of $V$. Let us take the $m$-dimensional subspace $V_m$ as the span of $\{e_j\}_{j=1}^m$. Consider the orthogonal projector of $V_m$ in $H$ as $P_m =P_{V_m}$. Then we can write $u^m_e=P_m u_e$ in the form $$u_e^m= \sum_{j=1}^m \alpha_j^m e_j,$$ which satisfies the following $m$-dimensional ordinary differential equation 
\begin{equation}\label{i3}
\nu \langle Au_e^m,v \rangle +\langle P_mB(u_e^m,u_e^m),v\rangle =\langle P_mf,v \rangle, \quad \forall v\in V_m.
\end{equation}
The existence of solution to \eqref{i3} will follow from the following lemma.
\begin{lem}[\cite{temam}, Lemma 1.4]\label{l3.3}
Let $X$ be a finite dimensional Hilbert space with scalar product $[\cdot,\cdot]$ and norm $[\cdot]$ and let $T$ be a continuous mapping $X$ into itself such that 
\begin{eqnarray}
[T(\xi),\xi] > 0 \text{  for  } [\xi]=k>0. \no
\end{eqnarray}
Then there exists $\xi \in X$, $[\xi] \leq k$, such that
\begin{equation}
T(\xi)=0. \no
\end{equation}
In our problem let us take  $X$ as $V_m$ and $T$ defined as
\begin{align*}
[T(u_e^m),v] = \nu (A^{1/2}u_e^m,A^{1/2}v ) +\langle P_m B(u_e^m,u_e^m),v\rangle -\langle P_mf,v \rangle.
\end{align*} 
Therefore,
\begin{align*}
[T(u_e^m),u_e^m] &= \nu |A^{1/2}u_e^m|^2 +\langle P_m B(u_e^m,u_e^m),u_e^m\rangle -\langle P_mf,u_e^m \rangle \\
&\geq \nu \|u_e^m\|^2 - \|f\|_{V'} \|u_e^m \| \\
&\geq \|u_e^m\|( \nu \|u_e^m \|- \|f\|_{V'}).  
\end{align*}
So if we choose $k= \|u_e^m \| > \frac{1}{\nu} \|f\|_{V'}$, we will get $[T(u_e),u_e^m]>0.$
\end{lem}
Thus by the Lemma \eqref{l3.3} we prove that there exists a solution of the Galerkin approximated system \eqref{i3}. 
Now let us take $v=u_e^m$ in \eqref{i3} to get
\begin{align*}
 \nu \|u_e^m\|^2 &\leq \frac{\nu}{2}\|f\|_{V'}^2 + \frac{1}{2\nu} \|u_e^m \|^2 \\
 \|u_e^m\|^2 &\leq  \frac{1}{\nu^2} \|f\|_{V'}^2. 
\end{align*}
Since $u_e^m$ is uniformly bounded in $V$, using Banach-Alaoglu theorem, we can extract a subsequence, still denoted by $u_e^m$ such that 
\begin{eqnarray}\label{i4}
u_e^m \rightharpoonup u_e \ \text{ in }\ V,
\end{eqnarray}
and by using compact embedding of $V$ in $H$ we get
\begin{eqnarray}\label{i5}
u_e^m \rightarrow u_e \ \text{ in }\ H.
\end{eqnarray}
Now we have to show that $u_e$ solves the weak formulation \eqref{i2}.
By using \eqref{i4} we get $\langle Au_e^m,v \rangle \rightarrow \langle Au_e,v \rangle, \quad \forall v \in V.$

To show $\langle P_mB(u_e^m,u_e^m),v\rangle \rightarrow \langle P_mB(u_e,u_e),v\rangle$, let us denote $P_m'=(I-P_m)$. We have
\begin{align*}
&|\langle P_mB(u_e^m,u_e^m),v\rangle - \langle P_mB(u_e,u_e),v\rangle| \\
&= |-i\sum_{n=m+1}^{\infty}(ak_{n+1}u_{e,n+2}u_{e,n+1}^*v_n^* + bk_nu_e^{n+1}u_{e,n-1}^* v_n^* - ck_{n-1}u_{e,n-1}u_{e,n-2}v_n^*)\phi_n| \\
&\leq \|P_m'v\||P_m'u_e|^2 \leq \|v\| |u_e-P_m u_e|^2 \rightarrow0.
\end{align*}
Thus we have convergence of each term in \eqref{i2} which implies $u^e$ satisfies weak formulation \eqref{i2}.

Now if we take $f \in H$. By taking $v$ as $Au_e$ in \eqref{i2} we can derive
\begin{align*}
\nu |Au_e|^2 &= -b(u_e,u_e,Au_e) + ( f, A u_e ) \\
&\leq C_1 \|u_e\||u_e||Au_e| + |f||Au_e| \\
&\leq \frac{\nu}{4} |Au_e|^2 + \frac{2C_1^2}{\nu} \|u_e\|^2|u_e|^2 + \frac{\nu}{4} |Au_e|^2 + \frac{2}{\nu} |f|^2 
\end{align*}
Therefore we get 
\begin{align*}
|Au_e|^2 \leq \frac{4C_1^2}{\nu^2} \|u_e\|^2|u_e|^2 + \frac{4}{\nu^2} |f|^2 < \infty.
\end{align*}
For the uniqueness let us define $\tilde{u}=u_e^1-u_e^2$, where $u_e^1$ and $u_e^2$ are the solution of \eqref{steady}. So $\tilde{u}$ satisfies 
\begin{equation}\label{i6}
\nu A \tilde{u} +B(u_e^1,u_e^1)- B(u_e^2,u_e^2)=0.
\end{equation}
Now taking inner product of \eqref{i6} with $\tilde{u}$ we get
\begin{align*}
\nu \|\tilde{u}\|^2 + b(u_e^1,u_e^1,\tilde{u})+ b(u_e^2,u_e^2,\tilde{u})&=0 \\
\nu \|\tilde{u}\|^2 +b(u_e^1,\tilde{u},\tilde{u})+ b(u,u_e^2,\tilde{u}) &=0 \\
\nu \|\tilde{u}\|^2 +b(u,u_e^2,\tilde{u}) &=0.
\end{align*}
For any solution of \eqref{steady} by taking duality with $u_e$ we get $$\|u_e\| \leq \frac{1}{\nu}\|f\|_{V'},$$ which implies
\begin{align*}
\nu \|\tilde{u}\|^2 &\leq C_1 \|u_e^2\|\|\tilde{u}\|^2 \\
(\nu^2-C_1 \|f\|_{V'})  \|\tilde{u}\|^2 &\leq 0.
\end{align*}
So we can conclude that if $\nu^2>C_1 \|f\|_{V'}$, then the stationary solution is unique.
\end{proof}

Now let us linearize the system \eqref{i1} around the solution $u_e$ of steady state system. The linearized system is given by
\begin{equation} \label{w=0}
\frac{d u}{dt}+\mathcal{A}u=B_1 U, \ u(0)=u^{0}.
\end{equation}
Here for each $u \in D(A)$, the operator 
\begin{eqnarray}
\mathcal{A}=A+B'(u_e), \ \text{ with } \ D(\mathcal{A})=D(A)
\end{eqnarray}
is closed, densely defined and $-\mathcal{A}$ generates a $C_0$-semigroup on $H$. 

\begin{lem}\label{lem3.4}
$-\A$ generates a $C_0$-analytic semigroup and the resolvant $(\lambda I-\A)^{-1}$ of the operator $\A$ is compact in $H$.
\end{lem}
\begin{proof}
For the proof we refer Proposition 3.1, \cite{barbu ns}.
\end{proof}
Therefore by Fredlhom-Riesz Theorem, $\A$ has a countable set of eigenvalues $\lambda_j$ and corresponding eigenvectors $\varphi_j$, i.e. 
$$\A \varphi_j= \lambda_j \varphi_j, \ \ j=1,2, \cdots.$$

We know $\A u = A u + B'(u_e)u = A u + B(u,u_e)+B(u_e,u)$ where the linear operator $A:D(A)\rightarrow H$ is defined through its action on the elements of the canonical basis of $H$ as
\begin{equation*}
A\phi_j=k_j^2\phi_j
\end{equation*}
where the eigenvalues $k_j^2$ satisfy relation \eqref{e2}. 

For $u,u_e \in H$ of the form  $u=\sum_{n=1}^{\infty}u_n\phi_n$ and  $u_e=\sum_{n=1}^{\infty}v_n\phi_n$, the bilinear operator $B(u,u_e)$ is defined as 
\begin{equation*}
B(u,u_e)=-i\sum_{n=1}^{\infty}(ak_{n+1}v_{n+2}
 u_{n+1}^*+bk_nv_{n+1}u_{n-1}^*+ak_{n-1}u_{n-1}v_{n-2}+bk_{n-1}v_{n-1}u_{n-2})\phi_n,
\end{equation*}
with the assumption $u_0=u_{-1}=v_0=v_{-1}=0.$ 

Therefore it can be seen easily that the eigenvalues of $\A$ are given by
\begin{align}
\lambda_j&= k_j^2 -i(ak_{n+1}v_{n+2}
 u_{n+1}^*+bk_nv_{n+1}u_{n-1}^*+ak_{n-1}u_{n-1}v_{n-2}+bk_{n-1}v_{n-1}u_{n-2}) \no \\ 
 &\qquad -i(ak_{n+1}u_{n+2} v_{n+1}^*+bk_nu_{n+1}v_{n-1}^*+ak_{n-1}v_{n-1}u_{n-2}+bk_{n-1}u_{n-1}v_{n-2}).
\end{align}
Observe that for each $j=1,2,  \cdots$, the $k_j^2$ are distinct and hence $\lambda_j$ are distinct too and $\sigma(\A)$ is semisimple. Moreover we denote the distinct eigenvectors $\phi_j$ corresponding to $\lambda_j$ as $\varphi_j$.

Let $\A^*$ be the dual operator of $\A$. The eigenvalues of $\A^*$ are $\{\bar{\lambda}_j\}_{j=1}^\infty$. As before $\{\bar{\lambda}_j\}_{j=1}^\infty$ are distinct and the corresponding eigenvectors are 
$$\A^* \varphi_j^*= \bar{\lambda_j} \varphi_j^*, \ \ j=1,2, \cdots.$$
Since $\lambda_j$ are distinct, for a given $\beta >0$, there exist only finite number of eigenvalues such that 
\begin{eqnarray} \label{lambda}
\cdots \geq Re \lambda_{N+1}>\beta>Re \lambda_{N}\geq\cdots \geq Re \lambda_{2} \leq Re \lambda_{1}.
\end{eqnarray}
Thus above discussion leads to following proposition.
\begin{prop}
Since the spectrum $\sigma(\A)$ is semi-simple, so there exists a bi-orthogonal system of eigenfunctions $\{\varphi_j\}_{j=1}^{\infty}$,  $\{\varphi_j^*\}_{j=1}^{\infty}$ such that 
\begin{eqnarray}
\langle \varphi_j, \varphi_i^* \rangle = \delta_{ij}, \quad i,j=1,\cdots,
\end{eqnarray}
and 
\begin{eqnarray}
\A \varphi_j= {\lambda_j} \varphi_j, \ \ j=1,2, \cdots, \qquad \A^* \varphi_j^*= \bar{\lambda_j} \varphi_j^*, \ \ j=1,2, \cdots.
\end{eqnarray}
\end{prop}
\begin{proof}
From the properties of the eigenvalues and eigenvectors of $\A$ and $\A^*$ proved above, proposition follows.
\end{proof}

In the next section our aim is to show that there exists a finite dimensional controller in the feedback form which will stabilize the linearized system. 
\subsection{Internal Stabilization of linearized system}
Let us denote  by $\mathcal{C}$ the following $N \times N$ matrix
\begin{eqnarray}\label{matrix}
\mathcal{C}= [\langle B_1 \varphi_j^*, \varphi_i^* \rangle]_{i=1 j=1}^{N \ \ N},
\end{eqnarray}
which will be useful in the proof of next Theorem. The precise characterization of finite dimensional controller which stabilizes the system is proved in the following theorem.
\begin{thm}\label{thm3.4}
Let $u^0\in H$. Then there exists a controller $U(t)$ of the form
\begin{eqnarray}\label{finite dim control}
U(t)= \sum_{j=1}^N a_j(t) \varphi_j^*, \qquad a_j\in L^2(0,\infty),
\end{eqnarray}
which stabilizes the system \eqref{w=0} with the exponent decay $-\beta$. Moreover the controller $a=\{a_j\}_{j=1}^N$ can be chosen in the feedback form 
\begin{eqnarray}
a_j(t)= -\langle \B_1 \varphi_j^*, \R_0 u^*(t)\rangle, \qquad j=1,\cdots,N, t\geq 0,
\end{eqnarray}
where $\mathcal{R}_0:D(\R_0) \subset H \rightarrow H$ is a riccati operator such that $\R_0=\R_0^*, \ \R_0 \geq 0$ and solves the riccati equation given in Theorem \ref{thm3.2}.
\end{thm}
\begin{proof}
Since the eigen values of $\A$ are semi-simple and distinct, so we can easily show that $\text{det} \ \mathcal{C}= \|\langle B_1 \varphi_j^*, \varphi_i^* \rangle\|_{i=1 j=1}^{N \ \ N} \neq 0$. So by Theorem 2.1 of \cite{barbu ns} we conclude that there exists a controller of the form \eqref{finite dim control}.

Let $\Sigma$ be the set of all eigen values of $\A$ and $\Sigma_N=\{\lambda_j\}_{j=1}^N$. Now we decompose the system \eqref{w=0} into two systems, one related to the unstable modes $\Sigma_N$ and the other to the stable modes $\Sigma\setminus \Sigma_N$. For, we write the space $H$ as the direct sum of two invariant subspaces of $\A$ related to $\Sigma_N$ and $\Sigma\setminus \Sigma_N.$ Let $\Gamma_N$ be a positively oriented curve enclosing $\Sigma_N$ but no other point of the spectrum of $\A$. Now let us take, $$H_N=lin \ span\{\varphi_j\}_{j=1}^N.$$
The operator $$P_N:H \rightarrow H_N$$ is defined by
$$P_N=\frac{1}{2\pi i} \int_{\Gamma_N}(\lambda I-\A)^{-1} \d \lambda,$$
where $\Gamma_N$ is a closed curve encloses the eigenvalues $\Sigma_N$.

We write the solution of the system \eqref{w=0} as $u=u_N+u_N^-,$ where $u_N=P_N u$, $u_N^-=(I-P_N)u$ and the operator $\A$ as $\A_{N}=P_N \A$, $\A_{N}^-=(I-P_N)\A$ (see chapter 3, \cite{kato} for the decomposition). Now, we rewrite the system \eqref{w=0} with the controller \eqref{finite dim control} by taking projection on $H_N $ as finite dimensional part,
 \begin{equation} \label{e22}
\frac{d u_N}{dt}+\A_{N}u_N = \sum_{j=1}^N a_j(t) P_NB_1\varphi_j^*, \ u_N(0)=P_Nu^{0},
\end{equation}
and on its orthogonal compliment as infinite dimensional part
 \begin{equation} \label{e23}
\frac{d u_N^-}{dt}+\A_{\N}^-u_N^- = \sum_{j=1}^N a_j(t) (I-P_N)B_1\varphi_j^*, \ u_N^-(0)=(I-P_N)u^{0}.
\end{equation}
Since the spaces $H_N=P_NH$ and $H_N^-=(I-P_N)H$ are invariant under $\A$, so we have $\sigma(\A_N)=\{\lambda_j\}_{j=1}^N$ and $\sigma(\A_N^-)=\{\lambda_j\}_{j=N+1}^\infty.$

We know from \eqref{lambda} that $-\A_N^-$ generates a $C_0$-analytic semigroup on $H_N^-$ and $\sigma(\A_N^-)= \{\lambda: Re \lambda>\beta\}$. This implies that
\begin{eqnarray}
\|e^{-\A_N^-t}\|_{\mathcal{L}(H,H)} \leq C e^{-\beta t}, \qquad \forall t\geq 0.
\end{eqnarray} 
Now we write the solution of the finite dimensional system \eqref{e22} as
$$u_N(t)=\sum_{j=1}^N u_j(t) \varphi_j.$$
Therefore by taking duality pairing with $\varphi_i^*$ for $i=1,\cdots,N$ with all the terms of the system \eqref{e22} we get, for $i=1,\cdots,N$, 
\begin{eqnarray}\label{e25}
\langle  \frac{d u_i(t)}{dt} , \varphi_i^*\rangle +\langle  (\mathbb{A} u(t))_i , \varphi_i^*\rangle= \sum_{j=1}^N a_j(t) \langle B_1\varphi_j^*,\varphi_i^*\rangle, \ u_i^0(0)=\langle u^{0},\varphi_i^*\rangle, 
\end{eqnarray}
where $\mathbb{A}$ is the diagonal matrix 
\begin{align*}
\mathbb{A}&=[\langle\A\varphi_j,\varphi_i^*\rangle]_{i,j=1}^N \\
&= 
  \begin{pmatrix}
    \lambda_1 & 0 & \dots & 0 \\
    0 & \lambda_2 & \dots & 0 \\
    \vdots & \vdots & \ddots & \vdots \\
    0 & 0 & \dots & \lambda_N
  \end{pmatrix}.
\end{align*}
So we have
\begin{eqnarray}\label{27}
\frac{d v}{dt}+\mathbb{A}v(t) = \mathcal{C}a(t), \ v(0)=v^0,
\end{eqnarray}
where $v(t)=\{u_i(t)\}_{i=1}^N$, $v_0=\{u^0_i\}_{i=1}^N$, $a(t)=\{a_i(t)\}_{i=1}^N$ and $\mathcal{C}= [\langle B_1 \varphi_j^*, \varphi_i^* \rangle]_{i=1 j=1}^{N \ \ N}$.

In the next Lemma \ref{l3.5} we will show that \eqref{e22} is exactly null controllable which implies it is exponentially stable i.e.
\begin{eqnarray}\label{e5}
|u_N(t)| \leq C e^{-\beta t}|P_N(u^0)| \leq  C e^{-\beta t}|u^0|,  \qquad \forall t\geq 0.
\end{eqnarray} 

Hence by Kalman controllability Theorem (see Theorem 2.1, \cite{ben4}) there exists a vector $a=\{a_j\}_{j=1}^N \subset L^2(0,T;\mathbb{C}^N)$ such that 
$$u_N(T)=0,$$
where $T>0$ is a fixed time. Without loss of generality,  we can assume that $a_j(t)=0, \ \forall \; t \geq T.$ 
Now, from \eqref{e23} by substituting the controller $U(t)= \sum_{j=1}^N a_j(t) \varphi_j^*$ and using the fact that $\sigma(\A_N^-)=\{\lambda_j\}_{j=N+1}^\infty$ and $-\A_N^-$ generates a $C_0$-semigroup we can deduce that
\begin{align}
|u_N^-(t)| &\leq  |e^{-\A_N^- t}||(I-P_N)u^0| +  \int_0^T |e^{-\A_N^- (t-s)}|\left( \sum_{j=1}^N |a_j(s)| \right) \d s \no \\
&\leq C e^{-\beta t}|u^0| + C \int_0^T e^{-\beta (t-s)}\left( \sum_{j=1}^N |a_j(s)| \right) \d s \no \\
&\leq C e^{-\beta t}|u^0|+Ce^{-\beta t}\|a\|_{L^2(0,T;\mathbb{C}^N)} , \qquad \forall t\geq 0, \no
\end{align}
which can be made less than $C e^{-\beta t}|u^0|$ by choosing the controller  
\begin{eqnarray}\label{28}
\int_0^T |a(t)|^2 \d t \leq C|P_Nu^0|^2 \leq C|u^0|^2.
\end{eqnarray}
Finally we obtain
\begin{eqnarray}\label{e6}
|u_N^-(t)| \leq   C e^{-\beta t}|u^0|,  \qquad \forall t\geq 0.
\end{eqnarray}
Therefore from \eqref{e5} and \eqref{e6} by adding the finite dimensional and infinite dimensional system we get
\begin{eqnarray}\label{e7}
|u(t)| \leq C e^{-\beta t}|u^0| , \qquad \forall t\geq 0.
\end{eqnarray}
\end{proof}
Now we will prove the lemma which we have used in the proof of the Theorem \eqref{thm3.4}.
\begin{lem}\label{l3.5}
The system \eqref{27} is exactly null controllable.  i.e. there exists a controller $a=\{a_j\}_{j=1}^N \subset L^2(0,T;\mathbb{C}^N)$ such that $u_i(T)=0, \ i=1, \cdots, N$ for a fixed $T>0$. 
\end{lem}
\begin{proof}
We know that by the Kalman controllability Theorem (Theorem 2.1, \cite{ben4}) finite dimensional system \eqref{e25} is exactly controllable iff 
\begin{eqnarray}\label{e29a}
\mathcal{C}^* e^{\mathbb{A}t}z=0, \qquad \forall t\geq 0 \Rightarrow z=0.
\end{eqnarray}
From the definition of $\mathcal{C}$ in \eqref{matrix} we have $\mathcal{C}^*=[\varphi_j^*,\B_1^* \varphi_i^*]_{i=1 j=1}^{N \ \ N}$ and 
$$
e^{\mathbb{A}t} = \begin{pmatrix}
    e^{\lambda_1t} & 0 & \dots & 0 \\
    0 & e^{\lambda_2t} & \dots & 0 \\
    \vdots & \vdots & \ddots & \vdots \\
    0 & 0 & \dots & e^{\lambda_Nt}
  \end{pmatrix} , z= \begin{bmatrix}
z_1 \\
\vdots \\
z_N
\end{bmatrix}_{1 \times N}.
$$
Therefore from \eqref{e29a} we get for each $i=1, \cdots, N$,
\begin{eqnarray*}
 c_{i1} e^{\lambda_1 t} z_1 + c_{i2} e^{\lambda_2 t}z_2 + \cdots +  c_{iN} e^{\lambda_1 t}  z_N = 0, \quad \forall t\geq 0,
\end{eqnarray*}
where $c_{ij}=\langle \varphi_i^* , \B_1 \varphi_j^* \rangle.$
This implies for each $i=1, \cdots, N$,
\begin{eqnarray*}
c_{i1} z_1 + c_{i2}z_2 + \cdots +  c_{iN} z_N = 0, \quad \forall t\geq 0.
\end{eqnarray*}
Since we know that $\mathcal{C}$ is a diagonal matrix we get $c_{ii}z_i=0, \quad \forall i=1, \cdots, N.$ Therefore we can conclude that $z_i=0, \quad \forall \; i=1, \cdots, N,$ which implies $z=0$.
\end{proof}

Now our aim is to write the controller in a feedback form which will stabilize the system \eqref{w=0}. 
\begin{thm}\label{thm3.2}
Let $\beta>0$ and $N$ be as defined in Theorem \eqref{thm3.4}. Then there exists a linear self-adjoint operator $\R_0:D(\R_0) \subset H \rightarrow H$ where $\R_0=\R_0^*, \ \R_0 \geq 0$ such that  
\begin{eqnarray}\label{i23}
b_1 |u^0|^2 \leq (\R_0u^0,u^0) \leq b_2 |u^0|^2, \quad \forall u^0 \in H,
\end{eqnarray}
for some constants $b_1,b_2 >0.$ Moreover 
\begin{eqnarray}\label{i12}
|\R_0u| \leq C\|u\|, \qquad \forall u\in V,
\end{eqnarray}
\begin{equation}\label{i13}
(\A u, \R_0u ) + \frac{1}{2} \sum_{j=1}^N (\B_1\varphi_j^*,\R_0u)^2 = \frac{1}{2} |A^{1/2}u|^2, \qquad u\in D(A).
\end{equation}
The feedback controller 
\begin{equation}
U(t)=  \sum_{j=1}^N (\B_1\varphi_j^*,\R_0u(t))\varphi_j^*
\end{equation}
exponentially stabilizes the linearized system \eqref{w=0}, i.e., the solution $u$ to corresponding closed loop system satisfies
\begin{eqnarray}\label{i20}
\int_0^T e^{2\beta t} |A^{1/2}u(t)|^2 \d t \leq C |u^0|^2 .
\end{eqnarray}
Moreover
\begin{eqnarray}\label{i32}
|u(t)| \leq Ce^{-\beta t}|A^{1/4}u^0|, \qquad u^0 \in H, \; \mbox{a.e.} \; t > 0.
\end{eqnarray}
\end{thm}
\begin{proof}
The proof is similar to \cite{barbu ns}, \cite{BT} which deals with stabilization of Navier Stokes equations. We associate a finite time horizon minimization problem as below:
\begin{align}
\varphi(u^0)=\inf_{U\in L^2(0,T;\mathbb{C}^N)}  \int_0^T |A^{1/2}u(t)|^2 + |U(t)|_N^2 \d t \label{i7} \\
\text{ subject to,   } \frac{du}{dt}+\A u - \beta u= \sum_{j=1}^N U_j B_1 \varphi_j^*, \ u(0)=u^0, \label{i8}
\end{align}
where $a_j$ is denoted by $U_j$ for $j=1,\cdots,N.$

Let us define $SU:=\sum_{j=1}^N U_j B_1 \varphi_j^*.$ 
We will first show that  $\forall u^0\in H$, $\varphi(u^0)<\infty. $ \\
For, from the Theorem \eqref{thm3.4} there exists an admissible pair $(u,U) \in  (C([0,T);H) \cap L^2_{\text{loc}}(0,T;D(A))) \times L^2(0,T;\mathbb{C}^N)$ which solves above optimal control problem. Thus $(u,U) $ is the optimal pair which solves \eqref{i7}.

By taking inner product of \eqref{i8} with $u$ and using the fact that $\A u=\nu Au+B(u,u_e)+B(u_e,u)$ we get
\begin{align*}
\frac{1}{2}\frac{\d}{\d t} |u|^2  + \nu |A^{1/2}u|^2 &\leq \beta |u|^2 +b(u,u_e,u) + (SU,u) \\
&\leq \beta |u|^2 + C_3 |u|\|Au_e\||A^{1/2}u| + \frac{1}{4}|U(t)|_N^2 +  \frac{1}{4}|u|^2 \\
\frac{1}{2}\frac{\d}{\d t} |u|^2  + \frac{\nu}{2} |A^{1/2}u|^2 &\leq  \left( \beta +\frac{2C_3^2}{\nu} |Au_e|^2+ \frac{1}{4} \right)|u(s)|^2 + \frac{1}{4}|U(t)|^2 .
\end{align*}
Integrating over $0$ to $T$ we get
\begin{align*}
 &|u(t)|^2 + \nu \int_0^T |A^{1/2} u(s)|^2 \d s \\
 &\qquad \leq |u^0|^2+ \int_0^T \left( 2 \beta +\frac{C_3^2}{\nu} |Au_e|^2+ \frac{1}{2} \right)|u(s)|^2 \d s+ \frac{1}{2} \int_0^T |U(s)|^2 \d s .
\end{align*}
Using Gronwall's Lemma we get
\begin{align*}
 &|u(t)|^2  \leq e^{CT}|u^0|^2+ \frac{1}{2} \int_0^T e^{Cs} |U(s)|^2 \d s \leq C|u^0|^2 , \ \forall t \in [0,T].
\end{align*}
So we have 
\begin{align}\label{i9}
\varphi(u^0) \leq b_2\;  |u^0|^2, \qquad \forall u^0\in H.
\end{align}
Moreover by using the property $ ( Au^0, u^0) \geq C | u^0|^2; \; \forall \; u^0 \in H $, we can prove that
\begin{align*}
|u^0|^2 \leq \int_0^T (|A^{1/2}u(t)|^2+|U(t)|_N^2) \d t.
\end{align*}
Thus,
\begin{align}\label{i10}
b_1|u^0|^2 \leq \varphi(u^0) , \qquad \forall u^0\in H.
\end{align}
Combining \eqref{i9} and \eqref{i10} we conclude
\begin{align}\label{i11}
b_1|u^0|^2 \leq \varphi(u^0)\leq b_2|u^0|^2 , \qquad \forall u^0\in H.
\end{align}
Therefore by using Theorem 3.1, \cite{ben4} we can conclude that there exists a linear self adjoint operator $\R_0:D(\R_0) \subset H \rightarrow H$ which is the Gateaux derivative of the function $\varphi(u^0)$ on $H$ such that
\begin{eqnarray} \label{i15}
\varphi(u^0) = \frac{1}{2}(\R_0u^0,u^0), \qquad \forall u^0\in H.
\end{eqnarray}
Let us take $ u^0 \in V $. So we have by Theorem 4, \cite{PBT} the solution $u \in C([0,T];V) \cap L^2(0,T;D(A))$ for any $T>0.$ By dynamic programming principle (see Barbu \cite{barbu}, Theorem 2.1), we know that for each $T>0$, the solution of \eqref{i7}-\eqref{i8} i.e. $(u^*,U^*)$ is also the optimal solution to the minmiization problem 
\begin{align}
\inf_{U\in L^2(0,T;\mathbb{C}^N)} \{ \frac{1}{2}\int_0^T (|A^{1/2}u(t)|^2 + |U(t)|_N^2) \d t + \varphi(u(T)), \text{subject to } \eqref{i8}\} \no.
\end{align}
Therefore by Pontryagin maximum principle we get
\begin{eqnarray}\label{e31}
U^*(t)=\{( p_T(t),\B_1\varphi_j^*)\}_{j=1}^N  , \ \text{ a.e. } t\in(0,T),
\end{eqnarray}
where $p_T$ is the solution of
\begin{eqnarray}\label{adjoint1}
-\frac{d z}{dt}+\A^*z -\beta z = -A u^*, \ z(T)=-\R_0 u^*(T).
\end{eqnarray}
We further conclude that 
\begin{equation}
\R_0 u^*(t)=-p_T(t), \qquad \forall t \geq [0,T].
\end{equation}
Hence from \eqref{e31}
\begin{equation}
U^*(t)=-\{( \R_0 u^*(t),\B_1 \varphi_j^*)\}_{j=1}^N, \qquad \forall t \geq [0,T].
\end{equation}
Now we will prove \eqref{i12}. For, if we take $u^0\in V$ then by existence theorem (see \cite{PBT}, Theorem 4) we can conclude that the optimal control $U^* \in L^2(0,T;H)$ and $u^* \in C([0,T];V) \cap L^2(0,T;D(A)) .$ Notice that \eqref{adjoint1} is a linear system. So an easy calculation gives that $z \in C([0,T];H) \cap L^\infty(0,T;V).$ Thus $z(0)= \R_0u^0 \in H$. Therefore by closed graph theorem we get $$|\R_0u|\leq \|u\|, \quad u\in V.$$
Now it is left to show that the operator $\R_0$ satisfies the algebraic riccati equation \eqref{i13}. By dynamic programming principle and \eqref{i15} we have 
\begin{align}\label{i16}
\frac{1}{2} (\R_0u^*(t),u^*(t)) = \varphi(u^*(t))= \frac{1}{2}\int_t^T (|A^{1/2}u^*(s)|^2 + |U^*(s)|_N^2) \d s, \qquad \forall t \geq 0, 
\end{align}
where $u^*(t) \in H$. Now differentiating \eqref{i16} and using self adjoint property of $\R_0$ on $H$ we get
\begin{align*}
( \R_0u^*(t),\frac{d u^*(t)}{d t}) = -\frac{1}{2} |A^{1/2}u^*(t)|^2 -\frac{1}{2} |U^*(t)|_N^2  .
\end{align*}
Using \eqref{i8} and \eqref{e31} we get
\begin{align}\label{i31}
&( \R_0 u^*(t),-\A u^*+\beta u  + U^*) + \frac{1}{2} |A^{1/2}u^*(t)|^2 + \frac{1}{2} |U^*(t)|_N^2  =0 \no \\
&(\R_0 u^*(t),-\A u^*+\beta u^* + U^* )+ \frac{1}{2} |A^{1/2}u^*(t)|^2  + \frac{1}{2} \sum_{j=1}^N ( \R_0 u^*(t),\B_1 \varphi_j^*)^2=0 \no \\
&( \R_0 u^*(t),\A u^* +\beta u^*) + \frac{1}{2} \sum_{j=1}^N ( \R_0 u^*(t),\B_1 \varphi_j^*)^2=\frac{1}{2} |A^{1/2}u^*(t)|^2 \d s,
\end{align}
for all $t\geq 0.$  
To prove \eqref{i20} and \eqref{i32} let us take the closed loop system
\begin{equation}\label{i18}
\frac{du}{dt} + \A u -\beta u +\sum_{j=1}^N ( \R_0 u,\B_1 \varphi_j^*)\B_1\varphi_j^*=0, \quad u(0)= u^0.
\end{equation}
Taking inner product of \eqref{i18} with $\R_0 u$ and using the riccati equation \eqref{i31} we get
\begin{align}
\frac{1}{2}\frac{d}{dt} (\R_0 u,u) &= ( \R_0 u,\frac{d u}{dt})\no \\
&=(\R_0u,\beta u) - (\R_0u,B(u)) -\frac{1}{2}\sum_{j=1}^N(\R_0 u, B_1\varphi_j^*)^2 -\frac{1}{2}|A^{1/2}u|^2 .
\end{align} 
This implies
\begin{align}
\frac{d}{dt} (\R_0 u,u)-2\beta (\R_0 u,u)+ |A^{1/2}u|^2 + \sum_{j=1}^N(\R_0 u, B_1\varphi_j^*)^2 \leq 0.
\end{align}
Integrating over $0$ to $t$ we get
 \begin{eqnarray}\label{i17}
(\R_0 u(t),u(t)) +  \int_0^t e^{-2\beta s} |A^{1/2}u(s)|^2 \d s \leq (\R_0 u^0,u^0) \leq b_2|u^0|^2.
\end{eqnarray}
Therefore we conclude
\begin{eqnarray}
\int_0^t e^{-2\beta s} |A^{1/2}u(s)|^2 \d s \leq C |u^0|^2,  \ \forall t \in [0,T].
\end{eqnarray}
From \eqref{i17} we further get
\begin{eqnarray}\label{i19}
|u(t)|^2 \leq \frac{1}{b_1}(\R_0 u(t),u(t)) \leq C e^{-2\beta t}| u^0|^2,  \  \forall t \in [0,T].
\end{eqnarray}
This completes the proof.
\end{proof}

\subsection{Internal stabilization of Nonlinear system}
\begin{thm}
The feedback controller 
\begin{equation}
U(t)=  -\sum_{j=1}^N (\B_1\varphi_j^*,\R_0(u-u_e))\varphi_j^*
\end{equation}
will exponentially stabilize the system \eqref{i1} to the solution $ u^e$ of the steady state system \eqref{steady} in a neighbourhood of $u^e$:
$$\mathcal{X}=\{u^0 \in H : \quad |u^0 - u^e | < \rho \}$$
of $u_e$ for some $\rho>0$.  
Moreover if $\rho$ is sufficiently small, then for each $u^0 \in \mathcal{X}$, the solution $u \in C([0,T);H) \cap L^2(0,T;V)$ to corresponding closed loop system 
\begin{equation}
\frac{du}{dt}+A u + B(u)+ \sum_{j=1}^N (\B_1\varphi_j^*,\R_0(u-u_e))\varphi_j^*= f , \ u(0)=u^0
\end{equation}
satisfies
\begin{eqnarray}
\int_0^T e^{2\beta t} |A^{1/2}(u(t)-u_e)|^2 \d t \leq C |(u^0-u_e)|^2,
\end{eqnarray}
and
\begin{eqnarray}
|(u(t)-u_e)| \leq Ce^{-\beta t}|(u^0-u_e)|, \qquad u^0 \in H.
\end{eqnarray}
\end{thm}

\begin{proof}
The proof is similar to Theorem 3.4 \cite{barbu ns}, Theorem 2.2 \cite{BT}.
Consider the system satisfied by $(u-u_e,u^0-u_e)$ but still denoted by $(u,u^0)$ we get
\begin{equation}\label{i21}
\frac{du}{dt}+ \nu A u + B'(u_e)u+B(u)+ \sum_{j=1}^N(\R_0 u, B_1\varphi_j^*)B_1 \varphi_j^*=0, \ u(0)=u^0.
\end{equation}
The problem reduces to proving the stability of the null solution of the closed loop system \eqref{i21}.
Next our aim is to show that $\varphi(u)= \frac{1}{2}(\R_0u,u)$ is a Lyapunov function for the system \eqref{i21} in a neighborhood of the origin.

By the Theorem 1.18, \cite{barbu ns} the system \eqref{i21} has at least one weak solution $u$ which is the limit of the strong solution $u_N$ to the system
\begin{equation}\label{i22}
\frac{du_N}{dt}+ \nu A u_N + B'(u_e)u_N +B_N(u_N)+ \sum_{j=1}^N(\R_0 u_N, B_1\varphi_j^*)B_1 \varphi_j^*=0, \ u_N(0)=u^0,
\end{equation}
where $B_N$ is the truncated operator
$B_N ( \cdot ):V \longrightarrow V^{'},$ 
\[B_N ( \cdot ) :=
  \begin{cases}
    B(u)      & \quad \text{if }\|u\| \leq N\\
    {\left( \frac{N}{ \|u\|} \right) }^{2}B(u)  & \quad \text{if } \|u\| > N.\\
  \end{cases}
\]
Similarly from Theorem 4.1 and Theorem 4.2 of \cite{BD} we can show that if $u^0 \in D(A)$ then $u_N \in \cap W^{1,\infty}_{\text{loc}}(0,\infty;H) \cap  L^{\infty}_{\text{loc}}(0,\infty;D(A))$ and if $u^0 \in V$ then $u_N \in W^{1,2}_{\text{loc}}(0,\infty;H) \cap  L^{2}_{\text{loc}}(0,\infty;D(A)) \cap C([0,T];V).$

Now from Theorem 1.18, \cite{barbu ns} we conclude $u_N \rightarrow u$ strongly in $L^2(0,T;H)$ and weakly in $L^2(0,T;V)$. Using the riccati equation \eqref{i18} in \eqref{i22} we obtain
\begin{align}\label{i29}
\frac{1}{2}\frac{d}{dt} &( \R_0 u_N,u_N) 
= \langle \R_0 u_N,\frac{d u_N}{dt}\rangle \no \\
=& -\beta (\R_0u_N,u_N) - (B_N(u_N),\R_0u_N) -\frac{1}{2}\sum_{j=1}^N(\R_0 u_N, B_1\varphi_j^*)^2  -\frac{1}{2}|A^{1/2}u_N|^2 .
\end{align} 
Now we have for all $t \geq 0$,
\begin{align}\label{i24}
(B_N(u_N),\R_0u_N)=b(u_N,u_N,\R_0u_N) &\leq \min \left(1,\frac{N^2}{\|u_N\|^2}\right) C_1 |u_N| \|u_N\||\R_0u_N| \no\\
&\leq C|u_N| \|u_N\|^2. 
\end{align}
By using \eqref{i23} we further get for each $t \geq 0$,
\begin{align}\label{i26}
(B_N(u_N),\R_0u_N) \leq C (\R_0 u_N,u_N)^{1/2}\|u_N\|^2 .
\end{align}
Using \eqref{i26} in \eqref{i29} we get
\begin{align}\label{i27}
\frac{d}{dt} (\R_0 u_N,u_N)+ 2\beta (\R_0u_N,u_N) +\frac{1}{2} |A^{1/2}u_N|^2 \leq \left(C (\R_0 u_N,u_N)^{1/2} -\frac{1}{2}\right) |A^{1/2}u_N|^2.
\end{align}
Now if we choose $\rho \leq (\frac{1}{2C})^2$, then $(\R_0 u_N,u_N) < \rho $. Then from \eqref{i23} we conclude $|A^{1/4}u^0|^2 < \rho,$ and from \eqref{i27} we get
\begin{eqnarray}\label{i28}
\frac{d}{dt} (\R_0 u_N,u_N)+ 2\beta (\R_0u_N,u_N) +\frac{1}{2} |A^{1/2}u_N|^2 \leq 0.
\end{eqnarray}  
By integrating over $0$ to $t$ we have
\begin{eqnarray}\label{i30}
e^{2\beta t}(\R_0 u_N(t),u_N(t)) + \frac{1}{2} \int_0^t e^{2\beta s} |A^{1/2}u_N(s)|^2 \d s \leq (\R_0 u^0,u^0) \leq C|u^0|^2,
\end{eqnarray}
for all $t \in [0,T]$. So we have
\begin{eqnarray}
\int_0^T e^{2\beta s} |A^{1/2}u_N(s)|^2 \d s \leq C|u^0|^2 < \rho.
\end{eqnarray}
So we will get a convergent subsequence $u_{N,n}$ such that $u_{N,n} \rightharpoonup u$ in $L^2(0,T;D(A^{1/2}).$ By using the fact that $u_{N} \rightarrow u$ strongly in $L^2(0,T;H)$, we can say $A^{1/2}u_{N,n} \rightarrow A^{1/2}u.$
Therefore we have
\begin{eqnarray}
\int_0^T e^{2\beta s} |A^{1/2}u(s)|^2 \d s \leq \liminf_n e^{2\beta t} |A^{1/2}u_{N,n}(t)|^2 \leq C|u^0|^2.
\end{eqnarray}
Further from \eqref{i30} and \eqref{i23} we get for all $t \in [0,T],$
\begin{eqnarray*} 
(\R_0 u(t),u(t)) \leq Ce^{-2\beta t} |u^0|^2 \\
|u(t)| \leq  Ce^{-\beta t} |u^0|.
\end{eqnarray*}
This completes the proof.
\end{proof}

\section{$H^\infty$ Stabilization}
In this section we study the $H^\infty$ stabilization problem for the sabra shell model namely,
 \begin{equation} \label{nonlinear}
\frac{d u}{dt}+\nu Au+B(u) =B_1 U+ B_2 w, \ u(0)=u^{0},
\end{equation}
where $A$ and $B$ are the operators as defined in section 2.  Note that the linear operator $A$ on $H= l^2( \C) $ is closed and densely defined with domain $D(A)$ and $B:H\rightarrow H$ is a nonlinear operator. Let us assume $B_1 \in \mathcal{L}(H,H)$ and $B_2 \in \mathcal{L}(H,H)$. 
 The linearised system around the steady state system $u_e$ is given by
 \begin{equation} \label{linear}
\frac{d u}{dt}+\mathcal{A}u=B_1 U+ B_2 w, \ u(0)=u^{0}.
\end{equation}

In this work our aim is to study the $H^{\infty}$ control problem corresponding to \eqref{nonlinear} which can be defined as finding a feedback operator $K \in \mathcal{L}(H,H)$ such that $A+K B_1$ is the infinitesimal generator of an exponentially stable semigroup on $H$. Moreover for a given $\gamma>0$ and for all $w \in \mathrm{L}^2(0,\infty;H)$ the solution to the closed loop system 
\begin{equation} \label{feedback}
\frac{d u}{dt}+Au+B(u)=B_1 K u+ B_2 w, \ u(0)=u^{0}
\end{equation}
obeys,
\begin{eqnarray}
\int_0^{\infty} \left(\|u\|^2 + \|Ku\|^2\right) \d t \leq \gamma \int_0^{\infty} \|w\|^2 \d t + \epsilon
\end{eqnarray}
for a given $\gamma>0$ and forall $w \in \mathrm{L}^2(0,\infty;H).$
   
 \subsection{Robust stabilization of the linearized equation:}

To study the robust feedback stabilization for the nonlinear system \eqref{nonlinear}, we are first going to study robust feedback stabilization of corresponding linearized system \eqref{linear}. To find robust feedback law for \eqref{linear} we have to solve folllowing control problem,
\begin{eqnarray}\label{ocp1}
&\sup_{w \in L^2(0,\infty;H)} \inf_{u \in L^2(0,\infty;H)} \{\mathcal{J}(u,U,w) \ \vert \ (u,U,w) \text{ satisfies } \eqref{linear}\}
\end{eqnarray}
where 
\begin{eqnarray}\label{e70}
\mathcal{J}(u,U,w) = \frac{1}{2} \int_0^{\infty} \|u\|^2 + \frac{1}{2} \int_0^{\infty}\|U\|^2 \d t  -\frac{\gamma}{2} \int_0^{\infty} \|w\|^2 \d t.
\end{eqnarray}

\begin{remark}
Note that in this section we consider  infinite time horizon problem problem. The stabilization of the finite time horizon problem proved in earlier section and the corresponding riccati operator $\R_0$ can be extended to infinite time horizon case using dynamic programming principle.  For more details please see Theorem 4.1 in \cite{raymond1} and lemma 3.5 in \cite{boudstab}. For notational convenience we denote riccati operator obtained for  the infinite time horizon problem again by   $\R_0$.  
  \end{remark}

We divide our problem in two steps, first we study the problem for a fixed $w \in L^2(0,\infty;H)$. We denote this minimization problem with initial condition $u^0$ and with $w$ as disturbance, by $P(u^0,w)$ i.e.
\begin{eqnarray}\label{ocp2}
\inf_{u \in L^2(0,\infty;H)} \{\mathcal{J}(u,U,w) \ \vert \ (u,U,w) \text{ satisfies } \eqref{linear}\}.
\end{eqnarray}
and then varing $w$, we take supremum over $w$.
In particular if we take $w=0$, we have the following theorem from the previous section.
\begin{thm}\label{thm4.2}
Let $w=0$ and $u^0\in H$. Then there exists a controller $U(t)$ which stabilizes the system \eqref{w=0} with the exponent decay $-\beta$. Moreover the controller can be chosen in the feedback form 
\begin{eqnarray}
U(t)=-\R_0 u(t), \qquad \forall t \geq 0,
\end{eqnarray}
where $\mathcal{R}_0 $ is a riccati operator $\mathcal{R}_0:D(\R_0) \subset H \rightarrow H$ such that $\R_0=\R_0^*, \ \R_0 \geq 0$. The operator $\R_0$ is the solution to the algebraic riccati equation 
\begin{equation}\label{rica}
\A^*\R_0 + \R_0\A - \R_0\B_1\B_1^*\R_0 + I=0.
\end{equation}
Furthermore, The optimal cost is given by
$$P(u^0,0)=\frac{1}{2} \left(u^0,\R_0u^0\right).$$
\end{thm}

Let us denote the optimal pair for the problem $P(u^0,0)$ by $(u_0,U_0)$ and recall $U_0=-\B_1^* \R_0 u_0$. Now we want to study the problem $P(u^0,w)$ for a fixed $w$.
\begin{thm}
Let the initial data $u^0 \in H$ and $w \in L^2(0,T;H)$. Then there exists a unique optimal pair  $(u_{u^0,w},\U_{u^0,w})$  such that the functional $ \mathcal{J}(u,U,w)$ attains its minimum at $(u_{u^0,w},U_{u^0,w})$. 
\end{thm}
\begin{proof}
We can easily prove it with arguments similar as in section 4 of \cite{BD}.
\end{proof}
To characterize the control $U_{u^0,w}$ we proceed as follows:
\begin{thm}
Let us define the operator $\A_{\R_0}:D(\A_{\R_0}) \subset H \rightarrow H$ by 
\begin{align}
D(\A_{\R_0})&=\{u \in H \vert (\A-\B_1\B_1^*\R_0 )u \in H\} \\
\A_{\R_0}u &= \A u- \B_1\B_1^*\R_0 u, \ \forall u\in D(\A_{\R_0}).
\end{align}
Then $\A_{\R_0}$ is the infinitesimal generator of an exponentially stable semigroup. The adjoint operator $((\A_{\R_0})^*,D((\A_{\R_0})^*))$ is given by,
$$D((\A_{\R_0})^*)=D(\A^*) \ \text{   and   } (\A_{\R_0})^*u = \A^*u-\R_0\B_1\B_1^*u \ \ \forall u \in D(\A^*).$$
\end{thm}
\begin{proof}
Since from Lemma \ref{lem3.4} we know that $\A$ generates a $C_0$ analytic semigroup and $\B_1$ is a linear operator, using Proposition 10 of \cite{BT 1} we have $\A_{\R_0}$ is the infinitesimal generator of an exponentially stable semigroup.

Further we also get from Proposition 2.4, Part2,  \cite{ben4} that $\A_{\R_0}^*$ is the infinitesimal generator of an exponentially stable semigroup.
\end{proof}
Now let us consider the coupled system,
\begin{align}
\frac{d u}{dt} + \A_{\R_0}u &= B_1\B_1^* p+ B_2 w, \ u(0)=u^{0}, \label{w fixed1} \\
 -\frac{dp}{dt} + \A_{\R_0}^* p&= \R_0 \B_2 w, \ p(\infty)=0. \label{w fixed2}
\end{align}
We prove the existence and uniqueness of the system \eqref{w fixed1}-\eqref{w fixed2}.
\begin{thm}\label{t4.5}
For all $u^0 \in H$, the system \eqref{w fixed1}-\eqref{w fixed2} has a unique solution $(u_w,p_w)  \in L^{2}(0,\infty;V)\cap C([0,\infty);H).$
\end{thm}
\begin{proof}
 $\A_{\R_0}^*$ generates an exponentially stable $C_0$ semigroup and $\R_0 \in \mathcal{L(H,H)}, \ \B_2w \in L^2(0,T;H)$, so using Proposition 3.1 of \cite{ben4} we get
\begin{align}\label{is1}
\|p_w\|_{L^{2}(0,\infty;V)\cap C([0,\infty);H)} \leq C\|\R_0 \B_2w\|_{L^2(0,T;H)} \leq C\|w\|_{L^2(0,T;H)}.
\end{align}
Therefore using the fact that $\A_{\R_0}$ generates $C_0$ analytic semigroup,  we conclude solution $u_w$ to \eqref{w fixed1} satisfies

\begin{align}\label{is2}
\|u_w\|_{L^{2}(0,\infty;V)\cap C([0,\infty);H)}  \leq C(|u^0| +\|w\|_{L^2(0,T;H)} ).
\end{align}
Adding \eqref{is1} and \eqref{is2} we get
\begin{align}\label{is3}
\|u_w\|_{L^{2}(0,\infty;V)\cap C([0,\infty);H)} + \|p_w\|_{L^{2}(0,\infty;V)\cap C([0,\infty);H)} \leq C(|u^0| +\|w\|_{L^2(0,T;H)} ).
\end{align}
Moreover it can be easily seen that the solution is unique.
\end{proof}

Now we will study the problem $P(u^0,w)$ given by  \eqref{ocp2} for a fixed $w\in L^2(0,\infty;H)$. Our aim is to prove that $(\B_1^* \R_0u_w+ B_1^* p_w)$ will be the minimizer of the problem $P(u^0,w).$
\begin{thm}\label{t1.6}
Let, $(u_w,p_w)$ be the solution of the system \eqref{w fixed1}-\eqref{w fixed2}. Then the solution of the optimal control problem \eqref{ocp2} is given by $(u_{w},\B_1^* \R_0u_w+ B_1^* p_w)$. 

Moreover $r_{u^0,w}=\R_0 u_w+p_w$ satisfies the system,
\begin{eqnarray}\label{r_w}
-\frac{dr}{dt}+\A^*r = u_{w},\ r(\infty)=0.
\end{eqnarray}
\end{thm}
\begin{proof}
Let us consider all  pairs $(u,U)$ that satisfy,
\begin{eqnarray}\label{e19}
\frac{du}{dt}+\A u = \B_1 U + \B_2w,\ u(0)=u^0.
\end{eqnarray}
Using algebric riccati equation \eqref{rica} in the cost functional \eqref{e70} we get,
\begin{align}
2 \J(u,U,w)&= \int_0^\infty \|u\|^2 + \|U\|^2 - \gamma \|w\|^2 \d t \nonumber \\
&= -\int_0^\infty 2(\R_0 u,\A u)\d t +\int_0^\infty \|\B_1^* \R_0 u\|^2 +\int_0^\infty \|U\|^2 \d t - \int_0^\infty \gamma \|w\|^2 \d t. \no
 \end{align}
From \eqref{e19} we have
\begin{align}
2 \J(u,U,w)&= \int_0^\infty 2\langle \R_0 u,u'\rangle \d t +\int_0^\infty 2(\R_0 u, -\B_1 U-\B_2 w) \d t+\int_0^\infty \|\B_1^* \R_0 u\|^2 \no \\
& \qquad+\int_0^\infty \|U\|^2 \d t - \int_0^\infty \gamma \|w\|^2 \d t \no \\
&= 2(\R_0 u^0,u^0) - 2 \int_0^\infty (\B_1^*\R_0 u, U) - 2 \int_0^\infty ( u,\R_0\B_2 w)+\int_0^\infty \|\B_1^* \R_0 u\|^2 \no \\
&\qquad+\int_0^\infty \|U\|^2 \d t - \int_0^\infty \gamma \|w\|^2 \d t.
\end{align} 
From \eqref{w fixed2}, putting the value of $\R_0 \B_2 w$ we get
\begin{align}
2 \J(u,U,w)&=  2(\R_0 u^0,u^0) - 2 \int_0^\infty (\B_1^*\R_0 u, U) -2 \int_0^\infty (u,-p_w'+\A_{\R_0}^*p_w)+\int_0^\infty \|\B_1^* \R_0 u\|^2 \no \\
&\qquad +\int_0^\infty \|U\|^2 \d t - \int_0^\infty \gamma \|w\|^2 \d t 
\end{align} 
Now by integration by parts we get
\begin{align}
2 \J(u,U,w)&
= 2(\R_0 u^0,u^0) - 2 \int_0^\infty (\B_1^*\R_0 u, U) - 2 \int_0^\infty (u',p_w) - (u^0,p_w(0))  \no \\
&\qquad -2 \int_0^\infty(\A_{\R_0}u,p_w)+\int_0^\infty \|\B_1^* \R_0 u\|^2+\int_0^\infty \|U\|^2 \d t - \int_0^\infty \gamma \|w\|^2 \d t .
\end{align}
Using the fact that $\A_{\R_0}=\A-\B_1 \B_1^* \R_0$ and \eqref{e19} we have
\begin{align}
&= 2(\R_0 u^0,u^0) - 2 \int_0^\infty (\B_1^*\R_0 u, U) - (u^0,p_w(0)) - 2 \int_0^\infty (u',p_w)- 2 \int_0^\infty (\A u,p_w)  \no \\
&\qquad + 2 \int_0^\infty (\B_1 \B_1^* \R_0 u,p_w)+\int_0^\infty \|\B_1^* \R_0 u\|^2+\int_0^\infty \|U\|^2 \d t - \int_0^\infty \gamma \|w\|^2 \d t \no \\
&= 2(\R_0 u^0,u^0) - 2 \int_0^\infty (\B_1^*\R_0 u, U) - (u^0,p_w(0)) - 2 \int_0^\infty (U,\B_1^*p_w) - 2 \int_0^\infty (\B_2w,p_w)\no \\
&\qquad + 2 \int_0^\infty (\B_1^* \R_0 u,\B_1^*p_w)+\int_0^\infty \|\B_1^* \R_0 u\|^2+\int_0^\infty \|U\|^2 \d t - \int_0^\infty \gamma \|w\|^2 \d t \no
\end{align}
 Now,
 \begin{align}
\int_0^\infty \|U-\B_1^* \R_0 u-\B_1^* p_w\|^2 \d t - \int_0^\infty \|\B_1^* p_w\|^2 \d t = -2\int_0^\infty(U,\B_1^* \R_0 u)  \no \\
- 2\int_0^\infty (U,\B_1^* p_w)
+ 2\int_0^\infty (\B_1^* \R_0 u,\B_1^* p_w)+\int_0^\infty \|U\|^2 \d t+\int_0^\infty \|\B_1^* \R_0 u\|^2
 \end{align}
Therefore we get,
\begin{align}
2 \J(u,U,w)= (\R_0 u^0,u^0) + \int_0^\infty \|U-\B_1^* \R_0 u-\B_1^* p_w\|^2 \d t - \int_0^\infty \|\B_1^* p_w\|^2  \no \\
-2(u^0,p_w(0))  - 2\int_0^\infty (B_2w,p_w) \d t- \int_0^\infty \gamma \|w\|^2 \d t
\end{align} 
So we get the optimal control for the problem \eqref{ocp2} as $U=\B_1^* \R_0 u_w + \B_1^*p_w$ (see Part 4, Lemma 4.4, \cite{zab}). 

Let us take $r_{u^0,w} = \R_0 u_{w}+p_w$, then easy calculation shows that $r_{u^0,w}$ solves 
\begin{equation}\label{e80}
-\frac{dr}{dt} + \A^* r = u_w,\ \ r(\infty)=0.
\end{equation}
\end{proof}
By substituting the value of the optimal control and optimal state from the Theorem \ref{t1.6},  in \eqref{e70},  we can write optimal cost corresponding to fixed $w$ as,
 \begin{align}\label{e26}
P(u^0,w) &= \frac{1}{2}\int_0^\infty \|u_w\|^2 \d t+\frac{1}{2} \int_0^\infty \|\B_1^* \R_0 u_w + \B_1^*p_w \| \d t -\frac{\gamma}{2} \int_0^\infty\|w\|^2\d t \no \\
&= \frac{1}{2}\int_0^\infty \|u_w\|^2 \d t+\frac{1}{2} \int_0^\infty \|\B_1^* r_{u^0,w} \| \d t-\frac{\gamma}{2} \int_0^\infty \|w\|^2\d t
 \end{align}
 
 The equations satisfied by $u_w$ and $ r_{u^0,w}$, lead to following lemma.
 \begin{lem}\label{l4.7}
 For all $w \in L^2(0,\infty;H)$,
 \begin{align}\label{e53}
 \int_0^\infty \|u_w\|^2 \d t+\int_0^\infty \|\B_1^* r_{u^0,w} \|^2 \d t = \int_0^\infty (\B_2 w,r_{u^0,w}) + (u^0,r_{u^0,w}(0)).
\end{align}
 \end{lem}
 \begin{proof}
We know that 
\begin{align}\label{e49}
\int_0^\infty (u_w',r_{u^0,w}) + \int_0^\infty (u_w,r_{u^0,w}') &= -(u^0,r_{u^0,w}(0))  .
\end{align}
If we put $r_{u^0,w} = \R_0 u_{w}+p_w$ in \eqref{w fixed1} we get
\begin{align}\label{e50}
\frac{du_w}{dt} + \A u_w = \B_1 \B_1^* r_{u^0,w} + \B_2w, \quad u_w(0)=u^0.
\end{align}
From \eqref{e50} we can write
\begin{align}\label{e51}
\int_0^\infty (u_w',r_{u^0,w}) = -\int_0^\infty (\A u_w,r_{u^0,w}) + \int_0^\infty (\B_1 \B_1^* r_{u^0,w},r_{u^0,w}) + \int_0^\infty (\B_2 w,r_{u^0,w}).  
\end{align}
From \eqref{e80} we have
\begin{align}\label{e52}
 \int_0^\infty (u_w,r_{u^0,w}') = \int_0^\infty (u_w,\A^* r_{u^0,w})-\int_0^\infty (u_w,u_w). 
\end{align}
Therefore using \eqref{e51} and \eqref{e52} in \eqref{e49} we get,
\begin{align*}
\int_0^\infty \|u_w\|^2 \d t + \int_0^\infty \|\B_1^* r_{u^0,w} \|^2 \d t = \int_0^\infty (\B_2 w,r_{u^0,w}) + (u^0,r_{u^0,w}(0)).
\end{align*}
\end{proof}

We now split $u_w$ in two parts, one solves homogeneous uncontrolled problem without disturbance  with initial data $u^0$ and other solves inhomogeneous equation with non zero disturbance but zero initial data. Let us denote $u_w=y_0+y_w$ where $y_0$ solves
\begin{equation}\label{e81}
\frac{dy}{dt} + \A_{\R_0} y = 0, \ y(0)=u^0,
\end{equation}
and $y_w$ solves
\begin{equation}\label{e82}
\frac{dy}{dt} + \A_{\R_0} y = B_1B_1^* p_w + B_2w, \ y(0)=0,
\end{equation}
where $p_w$ is the solution of \eqref{w fixed2}.
Let us set $\varphi_0=\R_0y_0$ and $\varphi_w= \R_0 y_w+p_w$. Therefore from \eqref{e80} we can write $r_{u^0,w}=\varphi_0+\varphi_w.$ 
Now above notation along with  \eqref{e26} and lemma \ref{l4.7} gives,
\begin{align}\label{e28}
P(u^0,w) &= \frac{1}{2}(u^0,r_{u^0,w}(0)) + \frac{1}{2}\int_0^\infty (\B_2 w,r_{u^0,w}) -\frac{\gamma}{2}\int_0^\infty \|w\|^2\d t \no \\
&= \frac{1}{2}(u^0,\varphi_0(0)) +  \frac{1}{2}(u^0,\varphi_w(0))  + \frac{1}{2}\int_0^\infty ( \B_2w,\varphi_0) \no \\
&\qquad  + \frac{1}{2}\int_0^\infty ( \B_2w,\varphi_w)-\frac{\gamma}{2}\int_0^\infty \|w\|^2\d t .
\end{align}
Our aim is to write this optimal cost as addition of optimal cost for problem with non zero initial data plus optimal cost for problem with zero initial data. Towards this aim, to estimate  second and third term in the above equation we use the following lemma. 
\begin{lem}\label{l3.11} 
For all $w \in L^2(0,\infty;H)$ we  have 
\begin{align}\label{e29}
(u^0,\varphi_w(0))=\int_0^\infty(w,\B_2^*\varphi_0),
\end{align}
where $\varphi_0$ solves
\begin{align}\label{e54}
-\frac{dz}{dt} + \A^* z = y_0, z(\infty)=0.
\end{align}
\end{lem}
\begin{proof}
We have from \eqref{w fixed2} and \eqref{e81}
\begin{align}\label{e55}
\int_0^\infty (\varphi_w',y_0) \d t &= \int_0^\infty (\A_{\R_0}^* \varphi_w,y_0) \d t -\int_0^\infty(\R_0 \B_2 w,y_0) \no \\
&= \int_0^\infty (\varphi_w,\A_{\R_0}y_0) \d t -  \int_0^\infty(\B_2 w,\varphi_0) \no \\
&= -\int_0^\infty (\varphi_w,y_0') \d t -  \int_0^\infty(\B_2 w,\varphi_0) \no.
\end{align}
Therefore integration by parts gives,
\begin{eqnarray*}
(u^0,\varphi_w(0))=\int_0^\infty(w,\B_2^* \varphi_0).
\end{eqnarray*}
\end{proof} 

Let us define the operator $T:H \times L^2(0,\infty;H)  \rightarrow \mathbb{C}$ as $$T(u^0,w)= \frac{1}{2} (u^0,\varphi_w(0)) + \frac{1}{2} \int_0^\infty(w,\B_2^*\varphi_0).$$
Therefore by Lemma \ref{l3.11} we can write $T(u^0,w)= (u^0,\varphi_w(0))=\int_0^\infty(w,\B_2^*\varphi_0).$

\begin{lem}\label{l4.10}
The operator $w \rightarrow T(u^0,w)$ is linear and we have
\begin{eqnarray}
|T(u^0,w)| \leq C |u^0| \|w\|_{L^2(0,\infty;H)}.
\end{eqnarray}
\end{lem}
\begin{proof}
Recall that from Theorem \ref{thm4.2} we know that $\phi_0=\R_0 y_0$ and the map from the initial data $u_0$ to the solution of \eqref{w fixed1} is continuous i.e. $y_0$ is continuous function of initial data. Therefore we have
\begin{align}
|T(u^0,w)| &\leq \int_0^\infty|(w,\B_2^*\varphi_0)| \leq \int_0^\infty|w||\B_2^* \R_0 y_0| \leq \int_0^\infty|w|| y_0| \no \\
& \leq C |u^0| \|w\|_{L^2(0,T;H)}.
\end{align}
\end{proof}
Therefore we can rewrite $P(u^0,w)$ in \eqref{e28} from Lemma \ref{l3.11} and above definition of operator $T$  as 
\begin{align}\label{e56}
P(u^0,w) =  \frac{1}{2}(u^0,\varphi_0(0)) + T(u^0,w) + \frac{1}{2}\int_0^\infty ( w,\B_2^*\varphi_w)-\frac{\gamma}{2}\int_0^\infty \|w\|^2\d t. 
\end{align}
Moreover, note that, $ \varphi_0(0) = \R_0 u^0$ and our characterisation in theorem \ref{thm4.2} allows us to write,
\begin{eqnarray}
P(u^0,w)= P(u^0,0) +T(u^0,w)+ P(0,w) ,
\end{eqnarray}
where,
$P(0,w)=\frac{1}{2}\int_0^\infty ( w,\B_2^*\varphi_w)-\frac{\gamma}{2}\int_0^\infty \|w\|^2\d t.$

Now we characterise $P (0, w) $. For, let us define the operator $Q : L^2(0,\infty;H) \rightarrow L^2(0,\infty;H)$ by $$Q(w)=\B_2^*\varphi_w , \qquad \forall w \in L^2(0,\infty;H).$$
\begin{lem}\textbf{[Properties of $Q$]}\label{l4.11}
\begin{enumerate}
    \item The operator $Q$ is linear and continuous.
    \item The operator $Q$ is positive and symmetric.
\end{enumerate}
\end{lem}
\begin{proof}
\begin{enumerate}
\item Since $\varphi_w$ satisfies the linear system \eqref{w fixed2}. So from Theorem \eqref{t4.5}, it follows that $Q$ is linear and continuous.
Moreover,
\begin{align}\label{e57}
\|\B_2^*\varphi_w\|_{L^2(0,\infty;H)} \leq \|\varphi_w\|_{L^2(0,\infty;H)} \leq \|w\|_{L^2(0,\infty;H)}.
\end{align}
Therefore $Q$ is a bounded operator.

\item Let us define $\gamma_0 = \sup_{\|w\|_{L^2(0,\infty;H)}=1} \left( w, Qw\right).$ 

Observe that the term $P(0,w)$ can be written as 
\begin{equation}
P(0,w)= \frac{1}{2} \int_0^\infty (w,Qw) - \frac{\gamma}{2} \int_0^\infty \|w\|^2.
\end{equation} 

 First we will show that $Q$ is positive. If we take $u^0=0$, then \ref{e81}  implies that $y_0=0$ and hence $\varphi_0 = \R u_0 =0$. Thus by putting $u^0=0$ in the Lemma \eqref{l4.7},  we get,
\begin{align}
\int_0^\infty (w,Qw)=\int_0^\infty (w,\B_2^* \varphi_w) =  \int_0^\infty \|u_w\|^2 \d t+\int_0^\infty \|\B_1 \varphi_w \|^2 \d t \geq 0.
\end{align}
Next to show that $Q$ is symmetric, let us take $w,v \in L^2(0,\infty;H)$ and $\varphi_w,\varphi_v$ be the corresponding solution of \eqref{w fixed2}.
We have
\begin{align}\label{e58}
\int_0^\infty (w,Qv)&=\int_0^\infty (w,\B_2^* \varphi_v)=\int_0^\infty (\B_2w, \varphi_v) \no \\
&=\int_0^\infty (\B_2w, \R_0 u_v+p_v) =  \int_0^\infty (\B_2w, \R_0 u_v)+  \int_0^\infty (\B_2w,p_v)\no \\
&=  \int_0^\infty (\R_0 \B_2w,  u_v)+  \int_0^\infty (\B_2w,p_v)\no \\
& = \int_0^\infty (-p_w' +\A_{\R_0}^* p_w ,u_v ) +  \int_0^\infty (\B_2w,p_v)\no \\
& = \int_0^\infty (p_w, u_v' +\A_{\R_0} u_v )+  \int_0^\infty (\B_1\B_1^* p_w,p_v) +  \int_0^\infty (\B_2w,p_v)\no \\
& = \int_0^\infty (p_w, \B_2 v )+  \int_0^\infty (\B_1^* p_w,\B_1^*p_v)+  \int_0^\infty (\B_2w,p_v).
\end{align}
Interchanging $v$ and $w$ we get
\begin{align}\label{e59}
\int_0^\infty (v,Qw) = \int_0^\infty (p_v, \B_2 w )+  \int_0^\infty (\B_1^* p_v,\B_1^*p_w)+  \int_0^\infty (\B_2v,p_w).
\end{align}
Therefore we get from \eqref{e58} and \eqref{e59}
$$\int_0^\infty (w,Qv)=\int_0^\infty (v,Qw).$$
Hence $Q$ is symmetric.
\end{enumerate}
\end{proof}

Now we will study the problem by taking supremum over $w$ of $P(u^0,w)$ i.e.
\begin{eqnarray}\label{e38}
P(u^0)=\sup_{w \in L^2(0,\infty;H)} P(u^0,w).
\end{eqnarray}
 First we will prove the existence of optimal $w$ and characterize the $w$ in terms of $\varphi_0$.
\begin{thm}
There exists $\gamma_0 >0$ such that if $\gamma>\gamma_0$, then the problem \eqref{e38} admits a unique solution. If $\gamma < \gamma_0$, then we have
$$\sup_{w \in L^2(0,\infty;H)} P(u^0,w)= \infty.$$
\end{thm}
\begin{proof}
	If $w=0$, then by Theorem \ref{thm4.2} there exists an optimal control for the problem $I(u^0,w)$. Therefore the set $\{I(u^0,w)| \ w \in L^2(0,\infty;H)\} $ is nonempty. 

Let us recall $\gamma_0 = \sup_{\|w\|_{L^2(0,\infty;H)}=1} \left( w, Qw\right)$.

Therefore using Lemma \eqref{l4.10} and Lemma \eqref{l4.11} in \eqref{e56} we have 
\begin{align*}
P(u^0,w) &\leq C|u^0|^2 +C |u^0| \|w\|_{L^2(0,\infty;H)} + \frac{1}{2} \gamma_0 \int_0^\infty |w|^2 \d t -\frac{\gamma}{2}\int_0^\infty |w|^2\d t \\
&\leq C|u^0|^2 +C |u^0| \|w\|_{L^2(0,T;H)} + \frac{(\gamma_0-\gamma)}{2}  \int_0^\infty |w|^2 \d t .
\end{align*}
Let us choose $\gamma>\gamma_0$ then we get $P(u^0,w)$ goes to $-\infty$ as $\|w\|_{L^2(0,\infty;H)} \rightarrow \infty$. Also we can see that as a function of $w$, $P(u^0,w)$ is a concave function and hence supremum over $w$ exists.  This ensures the existence of solution of the problem \eqref{e38}, when $\gamma >\gamma_0$.

Let us consider $\gamma <\gamma_0$. Then by the definition of $\gamma_0 = \sup_{\|w\|_{L^2(0,\infty;H)}=1} \left( w, Qw\right)$, there exists $w \in {L^2(0,\infty;H)}$ such that
\begin{eqnarray}
\frac{\gamma+\gamma_0}{2} < (w,Qw) < \gamma_0.
\end{eqnarray}
Now set $w_n=nw, \ \forall n\geq1$. Using Lemma \eqref{l4.10} and Lemma \eqref{l4.11} we deduce
\begin{align}\label{e60}
P(u^0,w_n)&= P(u^0,0) +T(u^0,w_n)+ P(0,w_n) \no \\
&= P(u^0,0) +T(u^0,w_n)+\frac{1}{2} \int_0^\infty (w_n,Qw_n) - \frac{\gamma}{2} \int_0^\infty \|w_n(t)\|^2 \d t \no \\
&=  P(u^0,0) +T(u^0,w_n)+\frac{n^2}{2} \int_0^\infty (w,Qw) - \frac{n^2\gamma}{2} \int_0^\infty \|w(t)\|^2 \d t \no \\
&\geq -C|u^0|^2 -nC |u^0| \|w\|_{L^2(0,\infty;H)} + \frac{n^2(\gamma_0-\gamma)}{2}  \int_0^\infty |w|^2 \d t .
\end{align} 
From \eqref{e60} as $n\rightarrow \infty$ we can conclude that $P(u^0,w_n)$ goes to $\infty$.
\end{proof}

Next our aim is to characterize the optimal disturbance.
\begin{thm}
Let us assume that $ \gamma > \gamma_0$.
Let $\hat{w}$ be an optimal disturbance for the problem \eqref{e38}, then $\hat{w}$ can be characterized as  $$-B_2 \hat{w} + \gamma \hat{w} + \B_2^* \varphi_0  = 0,$$ where $\varphi_0$ is the solution of the adjoint system 
\begin{equation}\label{p=0}
-\frac{d z}{dt} + \A^*z=0, \ \ z(\infty)=0.
\end{equation}
\end{thm}
\begin{proof}.
Let $\hat{w}$ be the optimal value for $P(u^0)$. Then, for $\lambda\in[0,1]$, we can deduce
\begin{align*}
P(u^0,\hat{w}+\lambda w) &- P(u^0,\hat{w})  \\
&=  \frac{1}{2}(u^0,\varphi_0(0)) + T(u^0,\hat{w}+\lambda w) + \frac{1}{2}\int_0^\infty ( \hat{w}+\lambda w,Q({\hat{w}+\lambda w})) \\
&\quad -\frac{\gamma}{2}\int_0^\infty \|\hat{w}+\lambda w\|^2\d t -\frac{1}{2}(u^0,\varphi_0(0)) - T(u^0,\hat{w})  \\
&\quad- \frac{1}{2}\int_0^\infty ( \hat{w},Q{\hat{w}})+\frac{\gamma}{2}\int_0^\infty \|\hat{w}\|^2\d t \\
&= T(u^0,\lambda w)  + \frac{1}{2}\int_0^\infty ( \hat{w}+\lambda w,Q({\hat{w}+\lambda w})) - \frac{1}{2}\int_0^\infty ( \hat{w},Q({\hat{w}+\lambda w})) \\
&\quad + \frac{1}{2}\int_0^\infty ( \hat{w},Q({\hat{w}+\lambda w}))- \frac{1}{2}\int_0^\infty ( \hat{w},Q{\hat{w}}) \\
&\quad - \frac{\gamma \lambda^2}{2}\int_0^\infty \|w\|^2\d t - \gamma\int_0^\infty(\hat{w},\lambda w) \no,
\end{align*}
which implies
\begin{align*}
P(u^0,\hat{w}+\lambda w) - P(u^0,\hat{w})&= T(u^0,\lambda w) + \frac{1}{2}\int_0^\infty (\lambda w,Q({\hat{w}+\lambda w})) + \frac{1}{2}\int_0^\infty ( \hat{w},Q({\lambda w}))\\
&\quad- \frac{\gamma \lambda^2}{2}\int_0^\infty \|w\|^2\d t - \gamma\int_0^\infty(\hat{w},\lambda w),
\end{align*}
dividing by $\lambda$ and taking limit as $\lambda$ goes to zero we get,
\begin{align}\label{e61}
\lim_{\lambda \rightarrow 0} \frac{P(u^0,\hat{w}+\lambda w) - P(u^0,\hat{w})}{\lambda} \geq 0 \no \\
\Rightarrow T(u^0,w) + \frac{1}{2}\int_0^\infty (w,Q\hat{w}) + \frac{1}{2}\int_0^\infty ( \hat{w},Q{w})- \gamma\int_0^\infty(\hat{w}, {w})  \geq 0 \no \\
\Rightarrow\int_0^\infty ( w,\B_2^*\varphi_0) + \int_0^\infty (w,Q \hat{w})- \gamma\int_0^\infty(w,\hat{w})  \geq 0,
\end{align}
where we have used the fact that the operator $Q$ is symmetric. Now by taking the Gateaux derivative in the direction of $-w$ we further get
\begin{align}\label{e62}
\int_0^\infty ( w,\B_2^*\varphi_0) + \int_0^\infty (w,Q \hat{w})- \gamma\int_0^\infty(w,\hat{w})  \leq 0.
\end{align}
Combining \eqref{e61} and \eqref{e62} we get 
\begin{align}\label{e63}
\int_0^\infty ( w,\B_2^*\varphi_0) +\int_0^\infty (w,Q \hat{w})- \gamma\int_0^\infty(w,\hat{w}) =0.
\end{align}
Since \eqref{e63} is true for all $w \in L^2(0,\infty;H)$ we get
\begin{align}\label{e64a}
-Q \hat{w} + \gamma \hat{w}=  \B_2^*\varphi_0 .
\end{align}
Let us define the operator $L:L^2(0,\infty;H) \rightarrow L^2(0,\infty;H)$ by
\begin{equation}
Lw=-Q w + \gamma w, \qquad \forall w\in L^2(0,\infty;H).
\end{equation}
Since, 
\begin{equation*}
(w,Lw) = (w,-Q w + \gamma w) \geq (-\gamma_0 + \gamma)  \|w\|_{L^2(0,\infty;H)}^2 \geq 0,
\end{equation*}
 $L$ is an isomorphism. Therefore from \eqref{e64a} the optimal disturbance is given by $\hat{w}$ as
\begin{equation}
\hat{w}= L^{-1}(\B_2^*\varphi_0 ),
\end{equation}
where $\varphi_0$ solves \eqref{p=0}.
\end{proof}

Now if we substitute $\hat{w}$ in the system \eqref{w fixed1}-\eqref{w fixed2}, we get the solution $(u_{\hat{w}},p_{\hat{w}})$. The corresponding optimal control is given by $\hat{U} =\B_1^*(\R_0u_{u^0,\hat{w}}+p_{\hat{w}})=\B_1^* r_{u^0,\hat{w}}.$
Let us denote $r_{u^0,\hat{w}}$ as $r_{u^0}$ for simplicity, since it only depends on the system with $w=0$. Therefore we get $\hat{U}=\B_1^* r_{u^0}$ and  $\hat{w}=-\frac{1}{\gamma}\B_2^* (\varphi_0+\varphi_{u^0})=-\frac{1}{\gamma}\B_2^* r_{u^0}$.

Let us introduce the operator $\R \in \mathcal{L}(H)$ defined by
$$\R:u^0 \rightarrow r_{u^0}(0).$$

 Moreover,  the optimal cost maximised over all disturbances is  given by 
\begin{equation}\label{e126}
 P(u^0,\hat{w})=\frac{1}{2}(u^0,\R u^0).
\end{equation}

\begin{lem}\label{l4.12}
The operator $\R\in \mathcal{L}(H,H)$ is symmetric and positive.
\end{lem}
\begin{proof}
To prove that $\R$ is positive we need to show that $\forall u^0 \in H$, $(\R u^0,u^0) \geq 0.$ But we have from \eqref{e126} that $(\R u^0,u^0) =2 P (u^0, \hat{w}).$ Since $\hat{w}$ is the solution of the supremum problem \eqref{e38}, we have $P (u^0, \hat{w}) \geq 0.$
Hence $\R$ is positive.

Now we will prove that $\R$ is symmetric. Let $r_{u^0}$ and $r_{v^0}$ be the solutions of \eqref{w fixed2} corresponding to initial conditions $u^0$ and $v^0$.
From the definition of $\R$ we get,
\begin{align*}
(\R u^0,v^0)&=(\R_0 u^0,v^0)+(r_{u^0}(0),v^0) = (u^0,\R_0v^0)+(r_{u^0}(0),v^0),
\end{align*}
since $\R_0$ is symmetric. We know that,
\begin{align}
(r_{u^0}(0),v^0)&= -\int_0^\infty ( r_{u^0}',u_{v_0}) \d t-\int_0^\infty ( r_{u^0},u_{v_0}') \no \\
&=\int_0^\infty (-\A_{\R_0}^* r_{u^0}+\R_0\B_2w_{y_0},u_{v^0}) \d t+\int_0^\infty ( r_{u^0},\A_{\R_0}u_{v^0} ) \no \\
&=\int_0^\infty (\R_0\B_2w_{y_0},u_{v^0}) \d t =\int_0^\infty (\B_2w_{y_0},\R_0u_{v^0}) \d t \no \\
&=\int_0^\infty (w_{y_0},\B_2^*p_{v^0}) \d t 
\end{align}
Since $Q$ is symmetric we get
\begin{align}
(r_{u^0}(0),v^0) = \int_0^\infty (\B_2^*p_{y_0},w_{v^0}) \d t = ({u^0},r_{v^0}(0))
\end{align}
Thus $\R$ is positive and symmetric.
\end{proof}

\begin{lem}
For all $t \geq 0$ we have
\begin{equation}
r_{u^0}(t)= \R \hat{u}(t).
\end{equation}
\end{lem}
\begin{proof}
The proof follows from Theorem 6.11.1 and Theorem 6.12.1  of \cite{LT}.
\end{proof}
 
Now onwards, for simplicity let us denote, $(u_{u^0},r_{u^0})$ by $(u,r).$
\begin{thm}\label{thm4.14}
Let $u^0 \in H$. Then the following system 
\begin{align*}
&\frac{du}{dt} + Au = B_1^* r- \frac{1}{\gamma} B_2^* r , \ u(0)=u^{0} \\ 
-&\frac{dr}{dt} +\mathcal{A}^*r =u , \ r(\infty)=0 \\ 
&r(t) = \R u(t), \qquad \forall t\geq 0,
\end{align*}
has a unique solution  
\begin{eqnarray*}
(u,r) \in \left(L^\infty(0,\infty;H) \cap L^2(0,\infty;V) \right) \times \left(C([0,\infty];H) \cap L^2(0,\infty;V) \right).
\end{eqnarray*}
It satisfies 
\begin{align*}
\|u\|_{L^\infty(0,\infty;H) \cap L^2(0,\infty;V)} + \|r\|_{C([0,\infty];H) \cap L^2(0,\infty;V)} \leq C|u^0|.
\end{align*}
\end{thm}
\begin{proof}
By substituting the optimal control and the optimal disturbance in \eqref{w fixed1}-\eqref{w fixed2}, the theorem follows.
\end{proof}

\begin{thm}
For $u^0 \in H$ the following equation 
\begin{equation}
\frac{du}{dt} + Au = B_1 B_1^* \R u- \frac{1}{\gamma} B_2 B_2^* \R u , \ u(0)=u^{0},
\end{equation}
admits a unique solution in $L^\infty(0,\infty;H) \cap L^2(0,\infty;V).$
\end{thm}

Next we have the following lemma where we will show that $\R$ satisfies the algebraic riccati equation.
\begin{lem}
The operator $\R\in \mathcal{L}(H)$ is the unique solution of the following algebric riccati equation
$$\R^*=\R,$$
$$\text{ for all } u^0 \in H \text{ we have } \R u^0 \in H \text{ and } |\R u^0|\leq C|u^0|,$$
\begin{eqnarray}\label{riccati R}
A^* \R + \R\A + \R \B_1 \B_1^* \R - \frac{1}{\gamma}\R \B_2 \B_2^* \R - I=0
\end{eqnarray}
\end{lem}
\begin{proof}
We have shown in Lemma \ref{l4.12} that $\R$ is symmetric.

Next we show the second condition.
From \eqref{e126} we get
\begin{align*}
\|u\|_{L^2([0,T],H)}^2 + \|U\|_{L^2([0,T],H)}^2 - \gamma \|w\|_{L^2([0,T],H)}^2 &\leq C|u^0|^2 \\ 
\|u\|_{L^2([0,T],H)}^2 + \|U\|_{L^2([0,T],H)}^2  &\leq \gamma \|w\|_{L^2([0,T],H)}^2 + C|u^0|^2 \\
 \|u\|_{L^2([0,T],H)}^2 + \|U\|_{L^2([0,T],H)}^2  &
 \leq C|u^0|^2.
\end{align*}
The last inequality follows from the fact that optimal disturbance is linearly dependent on $u^0$.
The riccati equation satisfied by $\R$ follows from Theorem \ref{thm4.14}.
\end{proof}

Now we consider the unbounded operator $(\A_{\R},D(\A_{\R}))$  defined by
$$D(\A_{\R})=\{u| \A u-\B_1\B_1^* \R u\in H\},$$
$$\A_{\R}u=\A u-\B_1\B_1^* \R u \ \text{ for all } u \in D(\A_{\R}).$$
\begin{prop}
The linear operator $(\A_{\R},D(\A_{\R}))$ generates an analytic and exponentially stable semigroup on $H$.
\end{prop}\label{p4.18}
\begin{proof}
Let us take the derivative of $(\R u(t),u(t))$ and integrating from $0$ to $T$, we have 
\begin{align*}
(\R u(T),u(T))-(\R u^0,u^0)&= \int_0^T (\R \A u(t),u(t)) \d t + \int_0^T (u(t),\A^* \R u(t)) \d t \\
&\qquad + 2\int_0^T (\R \B_1 U(t),u(t))+ 2\int_0^T (\R \B_2 w(t) ,u(t)) \d t.
\end{align*}
Using the fact that $\R$ satisfies the algebric riccati equation \eqref{riccati R} we get
\begin{align*}
(\R u(T),u(T))-(\R u^0,u^0)&= - \int_0^T (\R \B_1 \B_1^* \R u(t),u(t)) \d t + \frac{1}{\gamma}\int_0^T(\R \B_2 \B_2^* \R u(t),u(t)) \d t \\
&\qquad + \int_0^T |u(t)|^2 \d t+ 2\int_0^T (\R \B_1 U(t),u(t)) \\
&\qquad+ 2\int_0^T (\R \B_2 w(t),u(t)) \d t.
\end{align*}
We obatain
\begin{align*}
(\R u(T),u(T))-(\R u^0,u^0)&=  \int_0^T |U(t)-\B_1^* \R u(t)|^2 \d t +\gamma \int_0^T |w(t)+\frac{1}{\gamma} \B_2^* \R u(t)|^2 \d t \\
&\qquad- \int_0^T |u(t)|^2 \d t-\int_0^T |U(t)|^2 \d t + \gamma \int_0^T |w(t)|^2 \d t .
\end{align*}
Let us choose $U=\B_1^* \R u$ and $w=0$, therefore we can see that $u$ is the solution of 
\begin{eqnarray}
\frac{du}{dt} = Au - \B_1 \B_1^* \R u \ \text{ on } (0,T)\times \Omega, \ u(0)=u^0,
\end{eqnarray}
and we have 
\begin{align*}
(\R u(T),u(T))+ \int_0^T |\B_1^* \R u(t)|^2 \d t -\frac{1}{\gamma}\int_0^T |\B_2^* \R u|^2 \d t + \int_0^T |u(t)|^2 \d t =(\R u^0,u^0).
\end{align*}
Now taking the limit $T$ goes to $\infty$ we finally obtain,
\begin{align*}
\int_0^\infty |\B_1^* \R u(t)|^2 \d t-\frac{1}{\gamma}\int_0^\infty |\B_2^* u|^2 \d t + \int_0^\infty |u(t)|^2 \d t \leq C|u^0|^2.
\end{align*}
This completes the proof.
\end{proof}

\begin{thm}\label{thm4.18}
Let $u^0 \in H$. Then the following system 
\begin{align}\label{e64}
&\frac{du}{dt} +\mathcal{A}_{\R}u =- \frac{1}{\gamma}B_2 B_2^* \R u , \ u(0)=u^{0} \\ 
&-\frac{dr}{dt} +\mathcal{A}^*r =u , \ r(\infty)=0 \label{e84}
\end{align}
has a unique solution  
\begin{eqnarray}
(u,r) \in \left(L^\infty(0,\infty;H) \cap L^2(0,\infty;V) \right) \times \left(C([0,\infty],H) \cap L^2(0,\infty;V) \right).
\end{eqnarray}
It satisfies 
\begin{align}
\|u\|_{L^2(0,\infty;V)} + \|r\|_{L^2(0,\infty;V)} \leq C|u^0|.
\end{align}
\end{thm}
\begin{proof}
Observe that 
\begin{align}
\|- \frac{1}{\gamma}B_2 B_2^* \R u\|_{L^2(0,\infty;H))} \leq C_{\gamma}\|u\|_{L^2(0,\infty;V))}.
\end{align}
We know that from Proposition \eqref{p4.18} that $\A_{\R}$ generates a exponentially stable semigroup. It yields

\begin{equation}
u(t)= e^{-\A_{\R }t} u^0 -\frac{1}{\gamma} \int_0^\infty e^{-\A_{\R}(t-s)} B_2 B_2^* \R u(s) \d s .
\end{equation}
Therefore  from the first equation of \eqref{e64}, using Proposition 3.1 of \cite{ben4} we get
\begin{align}\label{e65}
\|u\|_{C([0,\infty],H) \cap L^2(0,\infty;V)} &\leq C_{\gamma} \| u\|_{L^2(0,\infty;V)} \leq C  |u^0|. 
\end{align}
Similarly we get from \eqref{e84}
\begin{align}\label{e66}
 \|r\|_{C([0,\infty],H) \cap L^2(0,\infty;V)} \leq C  |u^0| 
\end{align}
Adding \eqref{e65} and \eqref{e66} we get
\begin{align*}
\|u\|_{L^2(0,\infty;V)} + \|r\|_{L^2(0,\infty;V)} \leq C |u^0|.
\end{align*}
\end{proof}
 
\subsection{Robust stabilization of the linearized  system.}
Consider the following system
\begin{eqnarray}\label{e35}
\frac{du}{dt} +\A_{\R}u = \B_2w, \ u(0)=u^0.
\end{eqnarray}
\begin{thm}
If $w \in L^2(0,\infty;H)$ and $u^0 \in H$, then the unique solution of \eqref{e35} satisfies the following inequality,
\begin{align}
\int_0^\infty |u(t)|^2 \d t + \int_0^\infty |\B_1^* \R u(t)|^2 \d t \leq C|u^0|^2 + \gamma |w|^2.
\end{align}
\end{thm}
\begin{proof}
Let us take the derivative of $(\R u(t),u(t))$ and integrating from $0$ to $\infty$, we have 
\begin{align*}
-(\R u^0,u^0)&= \int_0^\infty (\A \R u(t),u(t)) \d t + \int_0^\infty (u(t),\R \A^* u(t)) \d t - 2\int_0^\infty |\B_1^* \R u(t)|^2 \d t \\
&\qquad + 2\int_0^\infty (\B_2 w(t) ,u(t)) \d t.
\end{align*}
Using the algebric riccati equation of $\R$, we obtain 
\begin{align*}
-(\R u^0,u^0)&= - \int_0^\infty |\B_1^* \R u(t)|^2 \d t - \frac{1}{\gamma}\int_0^\infty |\B_2^* u|^2 \d t - \int_0^\infty |u(t)|^2 \d t\\
&\qquad + 2\int_0^\infty (w(t) ,\B_2^* \R u(t)) \d t.
\end{align*}
Further we obtain,
\begin{align*}
(\R u^0,u^0)&= \int_0^\infty |\B_1^* \R u(t)|^2 \d t + \gamma\int_0^\infty |w - \frac{1}{\gamma}\B_2^* \R u(t)|^2 \d t + \int_0^\infty |u(t)|^2 \d t\\
&\qquad - \gamma \int_0^\infty |w(t)|^2 \d t .
\end{align*}
Finally we can deduce that,
 \begin{align*}
\int_0^\infty |u(t)|^2 \d t + \int_0^\infty |\B_1^* \R u(t)|^2 \d t \leq C|u^0|^2 + \gamma |w|^2.
\end{align*}
\end{proof}

\subsection{Robust stabilization of the nonlinear system.}
Next our aim is to show that the optimal control $\hat{U}$ will stabilize the nonlinear system \eqref{nonlinear}. We will first prove the following lemma.
\begin{lem}\label{lem4.20}
Let us take $g \in L^2(0,\infty;H)$. The system 
\begin{align}\label{e85}
&\frac{du}{dt} +\mathcal{A}_{\R}u =B_2 w + g, \ u(0)=u_{0} \\ 
&-\frac{dr}{dt} +\mathcal{A}^*r =u , \ r(\infty)=0 \no
\end{align}
has a unique solution  
\begin{eqnarray*}
(u,r) \in \left(L^\infty(0,\infty;H) \cap L^2(0,\infty;V) \right) \times \left(C([0,\infty],H) \cap L^2(0,\infty;V) \right),
\end{eqnarray*}
for all $u^0 \in H$ and $w \in L^2(0,\infty;H).$ Moreover it satisfies 
\begin{align*}
\|u\|_{L^\infty(0,\infty;H) \cap L^2(0,\infty;V)} + \|r\|_{L^\infty(0,\infty;H) \cap L^2(0,\infty;V)} \leq C(|u^0|+ \|w\|_{L^2(0,\infty;H)} + \|g\|_{L^2(0,\infty;H)}).
\end{align*}
\end{lem}
\begin{proof}
It follows from the Theorem \ref{thm4.18}. 
\end{proof}

\begin{thm}
There exist $\kappa_0>0$ and a nondecreasing function $\pi:\mathbb{R}^+ \rightarrow \mathbb{R}^+$ such that if $0<\kappa<\kappa_0$ and $|y_0| + \|w\|_{L^2(0,\infty;H)} \leq \pi(\kappa)$, then the nonlinear system 
\begin{equation}\label{nonlinear1}
\frac{du}{dt} +\mathcal{A}_{\R}u+B(u) =B_2 w, \ u(0)=u_{0}
\end{equation}
 has a unique solution
\begin{eqnarray}
u \in \left(C([0,\infty],H) \cap L^2(0,\infty;V)   \right).
\end{eqnarray}
and the solution $u \in \Sigma_\mu= \{u \in L^\infty(0,\infty;H) \cap L^2(0,\infty;V); \|u\|_{L^\infty(0,\infty;H)} \leq \kappa, \|u\|_{L^2(0,\infty;V)}  \leq \kappa \}, \ \forall t \geq 0.$
\end{thm}
\begin{proof}
Let $u^0 \in H$ be arbitary. Let us denote $\Upsilon: L^2(0,\infty;H)  \rightarrow L^2(0,\infty;H) $ and defined by
$$\Upsilon(g)=u,$$
where $u$ is the solution of the system \eqref{e85}. From the Theorem \ref{lem4.20} we know that $\Upsilon$ is Lipschitz function from $ L^2(0,\infty;H) $ to $ \left(C([0,\infty],H) \cap L^2(0,\infty;V) \right) $.
We can write the solution of the nonlinear system \eqref{nonlinear} as
$$u=\Upsilon(-\B(u)).$$
Let us set $$\Gamma(u)=-\B(u), \qquad \Lambda=\Upsilon\circ\Gamma.$$
Now our aim is to show that $\Lambda$ maps $\Sigma_\mu$ to itself and it is a contraction map.
We have 
\begin{align}
|\B(u)| &\leq C_1|u| \|u\| , \quad \forall u \in V. \no \\
|\B(u) - \B(v)| &\leq C_1 \left(|u|\|u-v\| + \|u-v\||v|\right), \quad \forall u,v \in V. \label{e71} 
\end{align}
Integrating \eqref{e71} over $0$ to $\infty$ we get,
\begin{align}
|\B(u) - \B(v)|_{L^2(0,\infty;H)} &\leq C_1 \left(\|u\|_{L^\infty(0,\infty;H)} \|u\|_{L^2(0,\infty;V)} + \|v\|_{L^\infty(0,\infty;H)} \|v\|_{L^2(0,\infty;V)}\right) \no \\
&\qquad \|u-v\|_{L^\infty(0,\infty;H)} \|u-v\|_{L^2(0,\infty;V)}. \label{e73}
\end{align}
Therefore from \eqref{e73} we get
\begin{align}\label{e47}
\|\Gamma(u)-\Gamma(v)\|_{L^2(0,\infty;H) \times L^2(0,\infty;H)} &\leq C_1 \left(\|u\|_{L^\infty(0,\infty;H)} \|u\|_{L^2(0,\infty;V)} + \|v\|_{L^\infty(0,\infty;H)} \|v\|_{L^2(0,\infty;V)}\right) \no \\
&\qquad \|u-v\|_{L^\infty(0,\infty;H)} \|u-v\|_{L^2(0,\infty;V)},
\end{align}
 for all $u,v \in L^\infty(0,\infty;H) \cap L^2(0,\infty;V).$ 
 Moreover we have
\begin{align}\label{e48}
\|\Gamma(u)\|_{L^2(0,\infty;H)} \leq C_1  \|u\|_{L^\infty(0,\infty;H)} \|u\|_{L^2(0,\infty;V)},
\end{align}
 for all $u\in L^\infty(0,\infty;H) \cap L^2(0,\infty;V).$ 
Now let us take $X=(L^\infty(0,\infty;H) \cap L^2(0,\infty;V) $, it implies 
$$\Sigma_\kappa = \{u \in X; \|u\|_{L^\infty(0,\infty;H)}\leq \kappa, \|u\|_{L^2(0,\infty;V)} \leq \kappa\}.$$
Let us choose $\pi(\kappa) = \frac{\kappa}{2C}$ and $\kappa<\kappa_0$ where $\kappa_0=\frac{1}{2CC_1}$. Using the Lemma \ref{lem4.20} and \eqref{e48} we derive 
\begin{align*}
\|\Lambda(u)\|_{L^2(0,\infty;H)}  &\leq  C(|u^0|+ \|w\|_{L^2(0,\infty;H)} + \|\Gamma(u)\|_{L^2(0,\infty;H)}) \\
& \leq C(|u^0|+ \|w\|_{L^2(0,\infty;H)} +C_1 \|u\|_{L^\infty(0,\infty;H)} \|u\|_{L^2(0,\infty;V)}) \\
&\leq C\pi(\kappa) +CC_1 \|u\|_{L^\infty(0,\infty;H) \cap L^2(0,\infty;V)} \\
& \leq C\pi(\kappa) +CC_1 \kappa^2 \leq \kappa.
\end{align*}

So we proved that $\Lambda$ maps $\Sigma_\kappa$ to itself. Now we are left to show that $\Lambda$ is a contraction map.

From the Lemma \ref{lem4.20} and \eqref{e47} we get
\begin{align}\label{e72}
\|\Lambda(u)-\Lambda(v)\|_{L^2(0,\infty;H)} &\leq C\|\Gamma(u)-\Gamma(v)\|_{L^2(0,\infty;H)} \no \\
& \leq CC_1 \left(\|u\|_{L^\infty(0,\infty;H)} \|u\|_{L^2(0,\infty;V)} + \|v\|_{L^\infty(0,\infty;H)} \|v\|_{L^2(0,\infty;V)}\right) \no \\
&\qquad \|u-v\|_{L^\infty(0,\infty;H)} \|u-v\|_{L^2(0,\infty;V)} \\
&\leq 2CC_1 \kappa^2 \|u-v\|_{L^\infty(0,\infty;H) \cap L^2(0,\infty;V)}. 
\end{align}
The operator $\Lambda$ is a contradiction in $\Sigma_\kappa$. Therefore the system \eqref{nonlinear1} has a unique solution $u \in \Sigma_\kappa$.
\end{proof}

\begin{thm}
If $w \in L^2(0,\infty;H)$ and $u^0 \in H$, then the unique solution of \eqref{nonlinear1} satisfies the following inequality,
\begin{align}
\int_0^\infty |u(t)|^2 \d t + \int_0^\infty |\B_1^* \R u(t)|^2 \d t \leq C|u^0|^2 + \gamma |w|^2 + 2C_1\kappa^3.
\end{align}
\end{thm}
\begin{proof}
We know that $u$ is the solution of 
\begin{eqnarray}
\frac{du}{dt}=\A_{\R} u + \B_2w + \B(u), \quad \text{ in } (0,\infty), \ u(0)=u^0.
\end{eqnarray}
Let us take the derivative of $(\R u(t),u(t))$ and integrating from $0$ to $\infty$, we have 
\begin{align*}
-(\R u^0,u^0)&= \int_0^\infty (\A \R u(t),u(t)) \d t + \int_0^\infty (u(t),\R \A^* u(t)) \d t - 2\int_0^\infty |\B_1^* \R u(t)|^2 \d t \\
&\qquad + 2\int_0^\infty (\B_2 w(t) ,u(t)) \d t + 2\int_0^\infty (\B(u)(t) ,u(t)) \d t.
\end{align*}
Using the algebric riccati equation of $\R$, we obtain 
\begin{align*}
-(\R u^0,u^0)&= - \int_0^\infty |\B_1^* \R u(t)|^2 \d t - \frac{1}{\gamma}\int_0^\infty |\B_2^* u|^2 \d t - \int_0^\infty |u(t)|^2 \d t\\
&\qquad + 2\int_0^\infty (w(t) ,\B_2^* \R u(t)) \d t + 2\int_0^\infty (\B(u)(t) ,u(t)) \d t.
\end{align*}
Further we obtain,
\begin{align*}
(\R u^0,u^0)&= \int_0^\infty |\B_1^* \R u(t)|^2 \d t + \gamma\int_0^\infty |w - \frac{1}{\gamma}\B_2^* \R u(t)|^2 \d t + \int_0^\infty |u(t)|^2 \d t\\
&\qquad - \gamma \int_0^\infty |w(t)|^2 \d t + 2\int_0^\infty (\B(u)(t) ,u(t)) \d t.
\end{align*}
We have 
\begin{align*}
 &\int_0^\infty |\B_1^* \R u(t)|^2 \d t + \gamma\int_0^\infty |w - \frac{1}{\gamma}\B_2^* \R u(t)|^2 \d t + \int_0^\infty |u(t)|^2 \d t- \gamma \int_0^\infty |w(t)|^2 \d t \\
 &\qquad \leq \|\R\|_{\mathcal{L}(H)}|u^0|^2+ 2\int_0^\infty |(\B(u)(t) ,u(t))| \d t.
\end{align*}
Finally we can deduce that,
 \begin{align*}
\int_0^\infty |u(t)|^2 \d t + \int_0^\infty |\B_1^* \R u(t)|^2 \d t &\leq C|u^0|^2 + \gamma |w|^2 + 2\int_0^\infty |(\B(u)(t) ,u(t))| \d t \\
&\leq C|u^0|^2 + \gamma |w|^2 + 2C_1|u|_{L^\infty(0,\infty;H)} \|u\|_{L^2(0,\infty;V)}^2 \\
&\leq C|u^0|^2 + \gamma |w|^2 + 2C_1 \kappa^3.
\end{align*}
This completes the proof.
\end{proof}

\end{document}